\documentclass[a4paper, 11pt]{article}%classe article, taille 12pt.

	\usepackage[francais,english]{babel} 				%utilise la langue francaise.
	\usepackage[utf8]{inputenc} 				%utilise le codage usuel utf-8.
	\usepackage{amsmath} 						%pour les symboles mathématiques.
	\usepackage{amsthm} 						%donne les environnements pour les théorèmes.
	\usepackage{amsfonts}
	\usepackage{makeidx} 						%pour faire la table d'index.

	\usepackage{listings} 						%pour insérer du code dans le texte. Voir la documentation.
	\usepackage{color}

	\usepackage{url}							%permet de mettre des liens url, \url, et des liens hyperref.
	\usepackage{hyperref}
	\usepackage{subfigure}
	\usepackage{graphicx}
	\usepackage{multicol}
	\usepackage{animate} 						%pour utiliser les outils avancés de bibliographie.
	\usepackage{pifont}

	\usepackage{appendix}

			\usepackage{algorithm}
			\usepackage{algorithmic}
			
						\usepackage{fullpage}
			\usepackage{authblk}
\makeindex

\begin{document}		%Contenu du document.

%\includegraphics[width = 50mm]{logo_jpg.jpg} 
%\hfill 
%\includegraphics[width = 60mm]{logo_oxford.png} 

%%%%%%%%%%%%%%%%%%%%%%%%%%%%%%%%%%%%%%%%%%%%%%%%%%%%%%%%%%%%%%
% Début de la définition des environnements pour les théorèmes. Le package amsthm permet de définir des environnements :

\theoremstyle{plain}			%%défaut
	\newtheorem{lem}		{Lemma}
	\newtheorem{prop}		{Proposition}
	\newtheorem{theo}		{Theorem}	%Théorème : nom affiché pour l'environnement.
	\newtheorem{coro}		{Corollaire}
\theoremstyle{definition}
	\newtheorem{defn}		{Definition}
	\newtheorem{conj}		{Conjecture}
	\newtheorem{exmp}		{Exemple}
\theoremstyle{remark}			%%En donnant un style appliqué à plusieurs types de "theorem"
	\newtheorem*{case}		{Cas}	%%Type non numéroté.
	\newtheorem*{mynote}		{Note}	%%Type numéroté.
	\newtheorem{remark}		{Remark}	%remarl : nom utilisé pour appelé l'environnement : \begin{remark}

% Fin de la définition des environnements pour les théorèmes.
%%%%%%%%%%%%%%%%%%%%%%%%%%%%%%%%%%%%%%%%%%%%%%%%%%%%%%%%%%%%%%%

\def\RR{\mathbb{R}}
\def\pa{\partial}

%%%%%%%%%%%%%%%%%%%%%%%%%%%%%%%%%%%%%%%%%%%%%%%%%%%%%%%%%%%%%%%
%\makeindex

%Titre, auteur, date :
%\maketitle									

%Tables des matières :
\newpage				%Insère une nouvelle page.

% Début de la défi
% Commandes de mises en pages :

\pagenumbering{arabic}	

\title{Incompressible limit of a continuum model of tissue growth for two cell populations}

 \author[1]{Pierre Degond}
  \author[1,2]{Sophie Hecht}
   \author[3]{Nicolas Vauchelet}
   \affil[1]{Department of Mathematics, Imperial College London, London SW7~2AZ, UK }
   \affil[2]{Francis Crick Institute, 1 Midland Rd, London NW1 1AT, UK }
   \affil[3]{LAGA, Universite Paris 13, 99 avenue Jean-Baptiste Clement, 93430 Villetaneuse, France }
   \affil[ ]{\textit {p.degond@imperial.ac.uk, sh5015@ic.ac.uk, vauchelet@math.univ-paris13.fr}}
   
   \maketitle
   
   \begin{abstract}
   This paper investigates the incompressible limit of a system modelling the growth of two cells population. The model describes the dynamics of cell densities, driven by pressure exclusion and cell proliferation. It has been shown that solutions to this system of partial differential equations have the segregation property, meaning that two population initially segregated remain segregated.
This work is devoted to the incompressible limit of such system towards a free boundary Hele Shaw type model for two cell populations.
   \end{abstract}
   
   \providecommand{\Acknowledgements}[1]{\textbf{{Acknowledgements. }} #1}
\Acknowledgements{PD acknowledges support by the Engineering and Physical Sciences
Research Council (EPSRC) under grants no. EP/M006883/1, EP/N014529/1 and
EP/P013651/1, by the Royal Society and the Wolfson Foundation through a
Royal Society Wolfson Research Merit Award no. WM130048 and by the National Science
Foundation (NSF) under grant no. RNMS11-07444 (KI-Net). PD is on leave
from CNRS, Institut de Math\'ematiques de Toulouse, France. SH acknowledge support from the Francis Crick Institute which receives its core funding from Cancer Re- search UK (FC001204), the UK Medical Research Council (FC001204), and the Well- come Trust (FC001204).
N.V. acknowledges partial support from the ANR blanche project Kibord No ANR-13-BS01-0004 funded by the French Ministry of Research. Part of this work has been done while N.V. was a CNRS fellow at Imperial College, he is really grateful to the CNRS and to Imperial College for the opportunity of this visit. PD, SH and NV would like to thanks Jean-Paul Vincent for stimulating discussion.}

\providecommand{\DataStatement}[1]{\textbf{{Data Statement. }} #1}
\DataStatement{No new data were collected in the course of this research.}

\providecommand{\keywords}[1]{\textbf{{Keywords.}} #1}

\providecommand{\AMS}[1]{\textbf{{AMS subject classifications.}} #1}
\AMS{35K55; 35R35; 65M08; 92C15; 92C10}

\section{Introduction}

Diversity is key in biology. It appears at all kind of level from the human scale to the microscopic scale, with million of cells types; each scales impacting on the others. During development, the coexistence of different cells types following different rules impact on the growth of tissue and then on the global structures. In a more specific case, this can be observed in cancerous tissue with the invasion of tumour cells in an healthy tissue creating a abnormal growth. Furthermore, cancerous cells are not playing all the same roles. They can be proliferative or quiescent depending of their positions, ages, \ldots To study the influence of these diverse cells on each others from a theoretical view, we introduce mathematical model for multiple populations. In this paper we are interesting in the global dynamics and interactions of the two populations, meaning that we focus specifically on continuous models. 

In the already existing literature on macroscopic model, we distinguish two categories. The most common ones involved partial differential equations (PDE) in which cells are represented by densities. These models have been widely used to model growth of tissue \cite{BD,RBEJPJ}, in particular for tumor growth \cite{AM04,BCGRS,Byrne,BenAmar}. Another way to model tissue growth is by considering free boundary models \cite{Cui,FriedmanHu,Greenspan}. In these models the tissue is described by a domain and its growth and movement are driven by the motion of the boundary. The link between these two types of model has been been made via an incompressible limit in \cite{HV,KimPozar,MPQ,PQTV,PQV,PV}. This link is interesting as both models have their advantages. On the one hand PDE relying models, also called mechanical models, are widely studied with many numerical and analytical tools. On the other hand free boundary models are closer to the biologic vision of the tissue and allow to study motion and dynamics of the tissue. This paper aims to extend the link between the mechanical and the free boundary models, in the case of multiple populations system.

In the specific case of multiple populations, several mathematical models have been already introduced. In particular in population dynamics, the famous Lotka-Volterra system \cite{Lotka} models the dynamics of a predator-prey system. This model has been extended to nonlinear diffusion Lotka-Volterra systems \cite{BERTSCH198756,B1987,BGHP,BT}. For the tumor growth modelling (see e.g. \cite{CGP}), some models focus on mechanical property of tissues such as contact inhibition \cite{BHIM,BPM,GSV} and mutation \cite{GALIANO}. They have been extended to multiple populations \cite{GALIANO,SHIGESADA197983}. Solutions to these models may have some interesting spatial pattern known as segregation \cite{BGHP,CarilloSchmidtchen,Mimura1980,SHIGESADA197983}.

The two cell populations system under investigation in this paper is an extension on a simplest cell population model proposed in \cite{BD,PQV}. Let $n(x,t)$ be the density of a single category of cell depending on the position $x \in \RR^d$ and the time $t>0$, and let $p(x,t)$ be the mechanical pressure of the system. The pressure is generated by the cell density and is defined via a pressure law $p=P(n)$. 
This pressure exerted on cells induces a motion with a velocity field $v=v(x,t)$ related to the pressure through the Darcy's law. The proliferation is modelled by a growth term $G(p)$ which is pressure dependent. 
With this assumption, the mathematical model reads
\begin{align*}
& \pa_t n + \nabla\cdot (n v) = n G(p),  \quad \mbox{ on }\ \RR^d\times\RR^+,  \\
& v = - \nabla p, \qquad p = P(n).
\end{align*}
In \cite{KimPozar,MPQ,PQTV,PQV,PV}, the pressure law is given by $P(n)= \frac{\gamma}{\gamma -1}n^{\gamma-1}$ which allows to recover the porous medium equation. 
However, in many tissues, cells may not overlap, implying that the maximal packing density should be bounded by $1$.
To take into account this non-overlapping constraint, the pressure law $P(n) = \epsilon \frac{n}{1-n}$ has been taken in \cite{HV}.
This latter choice of pressure law has also been taken in the present paper. For this one population model, it has been showed in \cite{HV}, that at the incompressible limit, $\epsilon \rightarrow 0$ (or $\gamma\to +\infty$ depending on the pressure expression), the model converges towards a Hele-Shaw type free boundary problem.

The previous model has the particularity to derive from the free energy 
\begin{equation*} %\label{eq:FE}
\mathcal{E}(n) = \int_\RR P(n(x))dx.
\end{equation*}
as a gradient flow for the Wasserstein metric. Using this property we derive a model for two species of cells. Let us denote $n_1(x,t)$
 and $n_2(x,t)$ the two cell densities depending on the position $x \in \RR^d$ and the time $t>0$. We assume that the  pressure depends on the total density $n=n_1 +n_2$. As the pressure depends on a parameter $\epsilon$, we introduce this dependancy in the notation. We define the free energy for the two cell populations by,
\begin{equation*} %\label{eq:FE2}
\mathcal{E}(n_{\epsilon}) = \int_\RR P({n_1}_{\epsilon}(x)+{n_2}_{\epsilon}(x))dx.
\end{equation*}
Restricting to the one dimensional case, the system of equation deriving from this free energy is then defined by,
\begin{align}
\partial_t {n_1}_{\epsilon} - \partial_x ({n_1}_{\epsilon} \partial_x p_{\epsilon}) = {n_1}_{\epsilon} G_1(p_{\epsilon}), \label{eq:n1} \\
\partial_t {n_2}_{\epsilon} - \partial_x ({n_2}_{\epsilon} \partial_x p_{\epsilon}) = {n_2}_{\epsilon} G_2(p_{\epsilon}), \label{eq:n2} \\
p_\epsilon = P(n_\epsilon) = \epsilon \frac{n_{\epsilon}}{1-n_{\epsilon}},  \label{eq:p2}  \\
n_\epsilon = {n_1}_\epsilon + {n_2}_\epsilon, \label{eq:n}
\end{align}
with $G_1, G_2$ the growth functions, and $p_{\epsilon}$ the pressure.  

The existence of solution for system \eqref{eq:n1}-\eqref{eq:n} has been proven in \cite{BHIM,BPM} for a compact domain $(-L,L)$ with $L>0$, with Neumann homogeneous boundary condition. In particular, it is shown that at a fix $\epsilon>0$, given initial conditions ${n_1}^{\mbox{\scriptsize ini}}_{\epsilon}$ and ${n_2}^{\mbox{\scriptsize ini}}_{\epsilon}$ satisfying, 
\begin{equation}\label{eq:segregini} 
\exists\, \zeta^0 \in \RR \mbox{ such that } {n_1}^{\mbox{\scriptsize ini}}_{\epsilon}= {n}^{\mbox{\scriptsize ini}}_{\epsilon} \mathbf{1}_{x \leq  \zeta^0} \mbox{ and } {n_2}^{\mbox{\scriptsize ini}}_{\epsilon}= {n}^{\mbox{\scriptsize ini}}_{\epsilon} \mathbf{1}_{x \geq  \zeta^0}, 
\end{equation}
and
\begin{equation}\label{eq:positivity} 
{n_1}^{\mbox{\scriptsize ini}}_{\epsilon},{n_2}^{\mbox{\scriptsize ini}}_{\epsilon} \geq 0 \mbox{ and } 0 < A_0 \leq {n_1}^{\mbox{\scriptsize ini}}_{\epsilon}+{n_2}^{\mbox{\scriptsize ini}}_{\epsilon} \leq B_0
\end{equation}
then there exists $\zeta_{\epsilon} \in C([0,\infty)) \cap C^1( (0,\infty))$ such that 
\begin{equation}\label{eq:segre} 
{n_1}_{\epsilon} (t,x)= {n}_{\epsilon} (t,x)\mathbf{1}_{x \leq  \zeta_{\epsilon}(t)} \quad \mbox{ and } \quad {n_2}_{\epsilon}(t,x)= {n}_{\epsilon}(t,x) \mathbf{1}_{x \geq  \zeta_{\epsilon}(t)} ,
\end{equation}
and ${n_1}_{\epsilon} $ and ${n_2}_{\epsilon} $ respectively satisfy \eqref{eq:n1} on $\{ (t,x), x\leq \zeta_{\epsilon}(t)\}$ and  \eqref{eq:n2} on $\{ (t,x), x\geq \zeta_{\epsilon}(t)\}$. In addition ${n}_{\epsilon} = {n_1}_{\epsilon} +{n_2}_{\epsilon} $ is solution to:
\begin{equation}\label{eq:solBPM}
  \left\{
      \begin{aligned}
         & \partial_t {n}_{\epsilon} - \partial_x ({n}_{\epsilon} \partial_x p_{\epsilon}) = {n}_{\epsilon} G_1(p_{\epsilon}) \quad \mbox{ on } \{ (t,x), x\leq \zeta_{\epsilon}(t)\},\\
        & \partial_t {n}_{\epsilon} - \partial_x ({n}_{\epsilon} \partial_x p_{\epsilon}) = {n}_{\epsilon} G_2(p_{\epsilon}) \quad \mbox{ on } \{ (t,x), x\geq \zeta_{\epsilon}(t)\},\\
        & n_\epsilon(t,\zeta_{\epsilon}(t)^-)=n_\epsilon(t,\zeta_{\epsilon}(t)^+), \\
        & \zeta_{\epsilon}'(t) = - \partial_x p(t,\zeta_{\epsilon}(t)^-)= - \partial_x p(t,\zeta_{\epsilon}(t)^+),\\
        & \partial_x {n}_{\epsilon}(\pm L,0) = 0\mbox{ for } t>0.
      \end{aligned}
    \right.
\end{equation}
 In \cite{BPM}, the reaction term is not the same than in this paper, however it is easy to see that there proof can be extended to our system under a set of assumptions for the growth functions which will be defined latter in this paper.
 
 The aim of this paper is is to study the incompressible limit $\epsilon \rightarrow 0$ for the two populations systems. When the two species are not in contact, the system is equivalent to the one population model \cite{HV}, this is why we limit ourself in this paper to the case where the two populations are initially in contact. To use the solutions defined in \cite{BPM}, we restrict the space to a compact domain $(-L,L)$ with $L>0$ and assume \eqref{eq:segregini} and \eqref{eq:positivity} are verified. Outside the domain $(-L,L)$, the system will be equivalent to the one population model.

We firstly remark that by adding \eqref{eq:n1} and \eqref{eq:n2}, we get,
\begin{align}
\partial_t n_{\epsilon} - \partial_x (n_{\epsilon} \partial_x {p}_{\epsilon}) = {n_1}_{\epsilon} G_1({p}_{\epsilon})+{n_2}_{\epsilon}G_2({p}_{\epsilon}) \mbox{ in } (-L,L). 
\label{eq:nn}
\end{align}
 Multiplying by $P'(n_\epsilon)$ we find an equation for the pressure,
\begin{equation}\label{eq:pp}
 \partial_t p_{\epsilon} - (\frac{p_{\epsilon}^2}{\epsilon}+p_{\epsilon})\partial_{xx} p_{\epsilon} -  |\partial_x p_{\epsilon}|^2 = \frac{1}{\epsilon}(p_{\epsilon}+\epsilon)^2({n_1}_{\epsilon} G_1(p_{\epsilon})+{n_2}_{\epsilon} G_2(p_{\epsilon})) \mbox{ in } (-L,L).
\end{equation}
Formally, passing at the limit $\epsilon \rightarrow 0$, we expect the relation,
$$
-p_0^2 \partial_{xx} p_0 = p_0^2 ({n_1}_{0} G_1(p_{0})+ {n_2}_{0} G_2(p_{0})) \mbox{ in } (-L,L).$$
In addition, passing formally to the limit $\epsilon\to 0$ into \eqref{eq:p2}, it appears clearly that $(1-n_0) p_0=0$. We consider the domain $\Omega_0(t)=\{x \in (-L,L), p_0(x,t)>0\}$, then, from the latter identity, $n_0 = 1$ on $\Omega_0$.
Moreover, from the segregation property, we have ${n_1}_{\epsilon}{n_2}_{\epsilon}=0$ when the two densities are initially segregated. 
Passing to the limit $\epsilon\to 0$ into this relation implies ${n_1}_{0} {n_2}_{0}=0$. Then we may split $\Omega_0(t)$ into two disjoint sets $\Omega_1(t)=\{x \in (-L,L),{n_1}_{0}(x,t)=1\}$ and $\Omega_2(t)=\{x\in (-L,L),{n_2}_{0}(x,t)=1\}$. Formally, it is not difficult to deduce from \eqref{eq:pp} that when $\epsilon\to 0$, we expect to have the relation
$$
-p_0^2 \partial_{xx} p_0 = \left\{
      \begin{aligned}
        p_0^2 G_1(p_0) \text{ on } \Omega_1(t),\\
        p_0^2 G_2(p_0) \text{ on } \Omega_2(t).\\
      \end{aligned}
    \right.
$$
Then we obtain a free boundary problem of Hele-Shaw type: On $\Omega_1(t)$, we have ${n_1}_{0}=1$ and $-\partial_{xx} p_0 = G_1(p_0)$, on $\Omega_2(t)$, we have ${n_2}_{0}=1$ and $-\partial_{xx} p_0 = G_2(p_0)$. 

The outline of the paper is the following. In Section \ref{sec:mainresult} we expose the main results of this paper, which are the convergence of the continuous model \eqref{eq:n1}-\eqref{eq:n} when $\epsilon \rightarrow 0$ to a Hele-Shaw free boundary model, and uniqueness for this limiting model. Section \ref{sec:proof} is devoted to the proof of these main results. The proof on the convergence relies on some a priori estimate and compactness techniques. We use Hilbert duality method to establish uniqueness of solution to the limiting system. Finally in Section \ref{sec:numeric}, we present some numerical simulations of the system  \eqref{eq:n1}-\eqref{eq:n} when $\epsilon$ is going to 0 and simulations of a specific application on tumor spheroid growth.

\section{Main results}\label{sec:mainresult}

In this paper we aim to prove the incompressible limit $\epsilon \rightarrow 0$ of the two populations model with non overlapping constraint \eqref{eq:n1}-\eqref{eq:n} in one dimension. We first introduce a list of assumptions on the growth terms and the initial conditions. For the growth, we consider the following set of assumptions:
\begin{equation}\label{hypG}
  \left\{
      \begin{aligned}
         &\exists\, G_m>0, \quad \| G_1 \|_{\infty} \leq G_m, \quad \| G_2 \|_{\infty} \leq G_m,\\
        & G_1', G_2' <0, \quad \mbox{ and }  \exists\, P^1_M, P^2_M>0, \quad G_1(P^1_M)=0 \mbox{ and } G_2(P^2_M)=0,\\
        & \exists\, \gamma>0, \quad \min(\inf_{[0,P^1_M]} |G_1'|,\inf_{[0,P^2_M]} |G_2'| )= \gamma,\\
        & P_M := \max(P_M^1,P_M^2), \quad \exists\, g_m \geq 0, \ \min\left(\inf_{[0,P_M]} G_1,\inf_{[0,P_M]} G_2\right) \geq -g_m. 
      \end{aligned}
    \right.
\end{equation}
The set of assumptions on the growth rate is standard and similar to the one in \cite{HV}. The parameters $P_M^1$ and $P_M^2$ are called homeostatic pressures which represent the maximal pressure that the tissue can handle before starting dying.
For the initial datas, we assume that there exists $\epsilon_0>0$ such that, for all $\epsilon\in (0,\epsilon_0)$,  for all $x\in(-L,L)$,
\begin{equation}\label{hypini}
  \left\{
      \begin{aligned}
      & 0 \leq {n_1}^{\mbox{\scriptsize ini}}_{\epsilon}, \quad 0 \leq {n_2}^{\mbox{\scriptsize ini}}_{\epsilon}, \quad {n}^{\mbox{\scriptsize ini}}_{\epsilon}={n_1}^{\mbox{\scriptsize ini}}_{\epsilon}+{n_2}^{\mbox{\scriptsize ini}}_{\epsilon}, \quad 0 < A_0 \leq {n}^{\mbox{\scriptsize ini}}_{\epsilon} \leq B_0, \quad \partial_x n_\epsilon^{\mbox{\scriptsize ini}} (\pm L) = 0, \\
      & \exists\, \zeta^0 \in (-L,L) \mbox{ such that } {n_1}^{\mbox{\scriptsize ini}}_{\epsilon}= {n}^{\mbox{\scriptsize ini}}_{\epsilon} \mathbf{1}_{x \leq  \zeta^0} \mbox{ and } {n_2}^{\mbox{\scriptsize ini}}_{\epsilon}= {n}^{\mbox{\scriptsize ini}}_{\epsilon} \mathbf{1}_{x \geq  \zeta^0}, \\
       & p^{\mbox{\scriptsize ini}}_{\epsilon}:= \epsilon \frac{n_\epsilon^{\mbox{\scriptsize ini}}}{1-n_\epsilon^{ini}} \leq P_M := \max(P_M^1,P_M^2), \\
       &  \max (\| \partial_{x} {n_1}^{\mbox{\scriptsize ini}}_{\epsilon} \|_{L^1(-L,L)}, \| \partial_{x} {n_2}^{\mbox{\scriptsize ini}}_{\epsilon} \|_{L^1(-L,L)}) \leq C,   \\
       & \exists\, n_1^{\mbox{\scriptsize ini}}, n_2^{\mbox{\scriptsize ini}} \in L_+^1(-L,L), 
\mbox{ such that } \|{n_1}^{\mbox{\scriptsize ini}}_\epsilon - n_1^{\mbox{\scriptsize ini}}\|_{L^1(-L,L)} \to 0  \\
& \qquad \mbox{ and } \|{n_2}^{\mbox{\scriptsize ini}}_\epsilon - n_2^{\mbox{\scriptsize ini}}\|_{L^1(-L,L)} \to 0, \mbox{ as } \epsilon \to 0.
%       & \exists\, K \subset \RR,\, K \mbox{ compact}, \quad \forall\, \epsilon \in (0,\epsilon_0), \quad \mbox{supp } n_\epsilon^{\mbox{\scriptsize ini}} \subset K.  \\
%       &  \partial_t n^{ini}_{\epsilon}:=\partial_x\cdot(n^{ini}\partial_x p^{ini})+n^{ini} G(p^{ini}) \geq 0, & \\
      \end{aligned}
    \right.
\end{equation}
These initial conditions imply that ${n_1}^{ini}_{\epsilon}$ and ${n_2}^{ini}_{\epsilon}$ are uniformly bounded in $W^{1,1}(-L,L)$. Notice also that the existence of $\zeta^0$ being the interface between the two species implies that the two populations are initially segregated.

From \cite{BPM}, we recover that at a fix $\epsilon>0$ under assumption \eqref{hypG}, given initial conditions ${n_1}^{\mbox{\scriptsize ini}}_{\epsilon}$ and ${n_2}^{\mbox{\scriptsize ini}}_{\epsilon}$ satisfying \eqref{hypini}, then there exists $\zeta_{\epsilon} \in C([0,\infty)) \cap C^1( (0,\infty))$ such that ${n_1}_{\epsilon} $ and ${n_2}_{\epsilon} $ verify \eqref{eq:segre} and ${n_1}_{\epsilon} $ and ${n_2}_{\epsilon} $ respectively satisfy \eqref{eq:n1} on $\{ (t,x), x\leq \zeta_{\epsilon}(t)\}$ and  \eqref{eq:n2} on $\{ (t,x), x\geq \zeta_{\epsilon}(t)\}$. In addition ${n}_{\epsilon} = {n_1}_{\epsilon} +{n_2}_{\epsilon} $ is solution to \eqref{eq:solBPM}.

\begin{remark}
Considering ${n_1}_\epsilon$ and ${n_2}_\epsilon$ defined previously, we have for $i=1,2$
$$ \partial_t  {n_i}_\epsilon= \partial_t  {n}_{\epsilon} (t,x)\mathbf{1}_{x \leq  \zeta_{\epsilon}(t)} + {n}_{\epsilon} \zeta_{\epsilon}'(t)\delta_{x=\zeta_{\epsilon}(t)} .$$
Given \eqref{eq:solBPM}, for all $\varphi \in C_c^{\infty} (-L,L)$ we compute, for $i=1, 2$
\begin{align*}
\int_{\RR} \partial_t {n_i}_{\epsilon} \varphi ~dx = & \int_{-\infty}^{\zeta_{\epsilon}(t)} \partial_t {n}_{\epsilon} \varphi ~dx + {n}_{\epsilon} (t, \zeta_{\epsilon}(t)) \zeta_{\epsilon}'(t) \varphi(\zeta_{\epsilon}(t))\\
= &  \int_{-L}^{\zeta_{\epsilon}(t)} \partial_x ({n}_\epsilon \partial_x p_\epsilon) \varphi ~dx + \int_{-L}^L {n_i}_\epsilon G_i(p_{\epsilon}) \varphi ~dx + {n}_{\epsilon}(t,\zeta_{\epsilon}(t)) \zeta_{\epsilon}'(t) \varphi(\zeta_{\epsilon}(t)) \\
= &  - \int_{-L}^{\zeta_{\epsilon}(t)} {n}_\epsilon \partial_x p_\epsilon \partial_x \varphi ~dx  +  {n}_\epsilon(t, \zeta_{\epsilon}(t)) \partial_x p_\epsilon(t, \zeta_{\epsilon}(t)) \varphi( \zeta_{\epsilon}(t)) \\
& +  \int_{-L}^L {n_i}_\epsilon G_i(p_{\epsilon}) \varphi ~dx - {n}_{\epsilon}(t, \zeta_{\epsilon}(t)) \partial_x p_\epsilon(t,  \zeta_{\epsilon}(t))\varphi( \zeta_{\epsilon}(t)) \\ 
= &  \int_{-L}^L {n_i}_\epsilon  \partial_x p_\epsilon \partial_x \varphi ~ dx  +  \int_{-L}^L {n_i}_\epsilon G_i(p_{\epsilon}) \varphi ~dx \\
= &  \int_{-L}^L(\partial_x ({n_i}_\epsilon  \partial_x p_\epsilon) + {n_i}_\epsilon G_i(p_{\epsilon}) )\varphi ~ dx .  
\end{align*}
Hence ${n_1}_\epsilon$ and ${n_2}_\epsilon$ are weak solutions to \eqref{eq:n1} and \eqref{eq:n2} on $(-L,L)$ respectively. This result will be used in the following.
\end{remark}

\medskip

Considering this particular solution, we are going to show the incompressible limit $\epsilon \rightarrow 0$ for system \eqref{eq:n1}-\eqref{eq:n}. The main result is the following
\begin{theo}\label{TH1}
Let $T>0$, $Q_T=(0,T)\times(-L,L)$. 
Let $G_1$, $G_2$ and $({n_1}^{\mbox{\scriptsize ini}}_{\epsilon})$, $({n_2}^{\mbox{\scriptsize ini}}_{\epsilon})$ satisfy assumptions \eqref{hypG}--\eqref{hypini}. After extraction of subsequences, the densities ${n_1}_{\epsilon}$, ${n_2}_{\epsilon}$ and the pressure $p_{\epsilon}$, solutions defined in \eqref{eq:segre}-\eqref{eq:solBPM}, converge strongly in $L^1(Q_T)$ as $\epsilon \rightarrow 0$ towards the respective limit ${n_1}_0, {n_2}_0 \in L^\infty([0,T];L^1(-L,L))\cap BV(Q_T)$, and
$p_0 \in BV(Q_T)\cap L^2([0,T];H^1(-L,L))$. Moreover, these functions satisfy, for all $(t,x)\in Q_T$,
\begin{align}
& 0 \leq {n_1}_0(t,x) \leq 1, \quad 0 \leq {n_2}_0(t,x) \leq 1,  \label{unif0}  \\
&  0 < A_0 e^{- g_m t} \leq n_0(t,x) \leq 1, \quad 0 \leq p_0 \leq P_M, \label{unif0np} \\
& \partial_t n_0 - \partial_{xx} p_0 = {n_1}_0 G_1(p_0)+{n_2}_0 G_1(p_0), \mbox{ in } \mathcal{D}'(Q_T), \label{limitn0}
\end{align}
where $n_0={n_1}_0 +{n_2}_0$, and
\begin{align}
\partial_t {n_1}_0 - \partial_x ( {n_1}_0 \partial_x p_0) = {n_1}_0 G_1(p_0), & \quad \text{ in } \mathcal{D}'(Q_T), \label{limitn1} \\
\partial_t {n_2}_0 - \partial_x ( {n_2}_0 \partial_x p_0) = {n_2}_0 G_2(p_0), & \quad \text{ in } \mathcal{D}'(Q_T), \label{limitn2}
\end{align}
complemented with Neumann boundary conditions $\partial_x p_0(\pm L)  = 0$.
Moreover, we have the relations
\begin{align}\label{n0p02}
(1-n_0)p_0=0,
\end{align}
and \begin{align}\label{segre}
{n_1}_0{n_2}_0=0,
\end{align}
and the complementary relation
\begin{align}\label{relcompbis}
p_0\left(\partial_{xx} P_0 + \int_0^t (n_{10}G_1(p_0) + n_{20} G_2(p_0))\,ds + n_0^{\mbox{\scriptsize ini}} -1\right) = 0.
\end{align}
where $P_0$ is defined by $P_0(t,x) =\int_0^t p_0(s,x)ds$.
\end{theo}

\begin{remark}
Introducing the set $\Omega_0=\{p_0>0\}$, we deduce that on $\Omega_0$ we have
$$
-\partial_{xx} P_0 = \int_0^t (n_{10}G_1(p_0) + n_{20} G_2(p_0))\,ds + n_0^{\mbox{\scriptsize ini}} -1.
$$
Deriving with respect to $t$, we find formally
$$
-\partial_{xx} p_0 = n_{10}G_1(p_0) + n_{20} G_2(p_0).
$$
We recognise the Hele-Shaw model.
Noticing also that taking $t=0$ into the relation \eqref{relcompbis}, we recover the expected relation
$p_0^{\mbox{\scriptsize ini}} n_0^{\mbox{\scriptsize ini}} = p_0^{\mbox{\scriptsize ini}}$.
\end{remark}

The proof of this convergence result is given in Section 3. It is straightforward to observe that adding \eqref{eq:n1} and \eqref{eq:n2} provides an equation on the total density similar to the one found in the one species case \cite{HV,PQV}. Then we use a similar strategy for the proof relying on a compatness method. However the presence of the two populations generate some technical difficulties. To overcome them, we use the segregation property. Notice that this paper is written in the specific case where the two species are separated by one interface, but could be generalised to many interfaces. Using the segregation of the species we are able to obtain a priori estimates on the densities, the pressure and their spatial derivatives. Compactness in time is deduced thanks to the Aubin-Lions theorem. The proof of convergence follows from these new estimates. However, the lack of estimates on the time derivative makes obtaining the complementary relation difficult, then we are not able to recover the usual relation but the one stated in \eqref{relcompbis} which may be seen as an integral in time of the usual one as explained in the above remark.

To complete the result on the asymptotic limit of the model, an uniqueness result for the Hele-Shaw free boundary model for two populations is provided in Proposition \ref{TH2} in \S\ref{sec:uniq}.
The proof of this uniqueness result for the limiting problem is based on Hilbert's duality method.

\section{Proof of the main results}
\label{sec:proof}

This section is devoted to the proof of Theorem \ref{TH1}, whereas in Section \ref{sec:uniq} the uniqueness of the solution to the Hele Shaw system is established. We first establish some a priori estimates.

\subsection{A priori estimates}

\subsubsection{Nonnegativity principle}

The following Lemma establishes the nonnegativity of the densities.
\begin{lem}\label{lem:nonneg2}
Let $({n_1}_\epsilon,{n_2}_\epsilon,p_\epsilon)$ be a solution to \eqref{eq:n1} and \eqref{eq:n2} such that ${n_1}_\epsilon^{\mbox{\scriptsize ini}}\geq 0$, ${n_2}_\epsilon^{\mbox{\scriptsize ini}}\geq 0$ 
and $ G_m <\infty$.
Then, for all $t\geq 0$, ${n_1}_\epsilon(t)\geq 0$ and ${n_2}_\epsilon(t)\geq 0$.
\end{lem}

\begin{proof}
To show the nonnegativity we use the Stampaccchia method. We multiply \eqref{eq:n1} by $\mathbf{1}_{{n_1}_\epsilon<0}$ and denote $|n|_- = \max(0,-n)$ for the negative part, we get
$$
\mathbf{1}_{{n_1}_\epsilon<0}\partial_t {n_1}_\epsilon - \mathbf{1}_{{n_1}_\epsilon<0}\partial_x ({n_1}_\epsilon \partial_x p_\epsilon) = \mathbf{1}_{{n_1}_\epsilon<0}{n_1}_\epsilon G_1(p_\epsilon). 
$$
With the above notation, it reads
$$
\partial_t |{n_1}_\epsilon|_{-} -\partial_x (|{n_1}_\epsilon|_{-} \partial_x p_\epsilon) = |{n_1}_\epsilon|_{-} G_1(p_\epsilon).
$$
We integrate in space, using assumption \eqref{hypG} and $\partial_x p_\epsilon (\pm L,t) = p'_{\epsilon}(n_\epsilon) \partial_x n_\epsilon (\pm L,t)=0$, we deduce
$$
\frac{d}{dt} \int_{-L}^L |{n_1}_{\epsilon}|_{-}dx \leq \int_{-L}^L |{n_1}_{\epsilon}|_{-} G_1(p_{\epsilon})dx \leq G_m\int_{-L}^L |{n_1}_{\epsilon}|_{-}dx. 
$$
Then we integrate in time,
$$
\int_{-L}^L |{n_1}_{\epsilon}|_{-}\,dx \leq e^{G_mt}  \int_{-L}^L |{n_1}_{\epsilon}^{\mbox{\scriptsize ini}}|_{-}\,dx.
$$
With the initial condition ${n_1}_{\epsilon}^{\mbox{\scriptsize ini}} > 0$ we deduce ${n_1}_{\epsilon} > 0$. With the same method we can show that if ${n_2}_{\epsilon}^{\mbox{\scriptsize ini}} > 0$ we have ${n_2}_{\epsilon}> 0$.
\end{proof}

\begin{remark} We notice that the positivity gives a formal proof of the segregation of any solution of \eqref{eq:n1}-\eqref{eq:n}.
Indeed, defining $r_{\epsilon}={n_1}_{\epsilon} {n_2}_{\epsilon}$ and multiplying \eqref{eq:n1} by ${n_2}_{\epsilon}$, \eqref{eq:n2} by ${n_1}_{\epsilon}$ and adding, we obtain the following equation for $r_{\epsilon}$,
$$
\partial_t r_{\epsilon} - \partial_x r_{\epsilon}\, \partial_x p_{\epsilon} - 2 r_{\epsilon} \partial_{xx} p_{\epsilon} = r_{\epsilon}(G_1(p_{\epsilon})+G_2(p_{\epsilon})).
$$
Given that $r^{\mbox{\scriptsize ini}}_\epsilon =0$, we get that $r_\epsilon=0$ at all time.
\end{remark}

\subsubsection{A priori estimates}

To show the compactness result we establish a priori estimate on the densities, pressure and their derivatives. We first compute the equation on the total density. As shown earlier ${n_1}_\epsilon$ and ${n_2}_\epsilon$ are respectively weak solutions of \eqref{eq:n1} and \eqref{eq:n2}. By summing the two equations we deduce that ${n}_\epsilon$ is a weak solution of \eqref{eq:nn}. Notice that this equation can be rewritten as,
\begin{equation}\label{eq:nH}
\partial_t n_{\epsilon} -\partial_{xx} H(n_{\epsilon})= {n_1}_{\epsilon} G_1(p_{\epsilon}) + {n_2}_{\epsilon} G_2(p_{\epsilon}),
\end{equation}
with  $H(n)=\int_0^{n} u P'(u)du = P(n)-\epsilon \ln (P(n)+\epsilon) + \epsilon \ln \epsilon$. 

We establish the following a priori estimates
\begin{lem}\label{lem:estim}
Let us assume that \eqref{hypG} and \eqref{hypini} hold. 
Let $({n_1}_\epsilon,{n_2}_\epsilon,p_\epsilon)$ be a solution to \eqref{eq:n1}--\eqref{eq:n}.
Then, for all $T>0$, and $t\in (0,T)$, 
we have the uniform bounds in $\epsilon\in(0,\epsilon_0)$,
\begin{align*}
& {n_1}_{\epsilon}, {n_2}_{\epsilon} \mbox{ in } L^\infty([0,T];L^1\cap L^\infty(-L,L)); \\
& 0\leq p_\epsilon \leq P_M, \qquad 0  < A_0 e^{-g_m t}\leq n_\epsilon(t) \leq \frac{P_M}{P_M+\epsilon} \leq 1.
\end{align*}
Moreover, we have that $({n_1}_\epsilon)_\epsilon$ and $({n_2}_\epsilon)_\epsilon$ are
uniformly bounded in $L^\infty([0,T],W^{1,1}(-L,L))$ and $(p_\epsilon)_\epsilon$ 
is uniformly bounded in $L^1([0,T],W^{1,1}(-L,L))$.
\end{lem}

\begin{proof}

{\bf Comparison principle}.

The usual comparison principle is not true for this system of equations. However we are able to show some comparison between the total density and $n_M$ defined by $n_M= \frac{P_M}{\epsilon+P_M}$ where $P_M$ is defined in \eqref{hypini}.
We deduce from \eqref{eq:nH} that
\begin{align*}
\pa_t (n_\epsilon - n_M) -\partial_{xx} (H(n_\epsilon)-H(n_M)) & \leq  {n_1}_\epsilon G_1(P(n_\epsilon)) - n_M  \mathbf{1}_{x\leq  \zeta_{\epsilon}(t)}  G_1(P_M) \\
& + {n_2}_\epsilon G_2(P(n_\epsilon)) - n_M \mathbf{1}_{x\geq  \zeta_{\epsilon}(t)}  G_2(P_M),
\end{align*}
where we use the monotonicity of $G_1$ and $G_2$ from assumption \eqref{hypG}.

Notice that, since the function $H$ is nondecreasing, the sign of $n_{\epsilon}-m_{\epsilon}$ is the same as the sign of $H(n_{\epsilon})-H(m_{\epsilon})$. Moreover,
$$
\partial_{xx} f(y) = f''(y) |\partial_x y|^2 +f'(y) \partial_{xx} y, 
$$
so for $y=H(n_{\epsilon})-H(n_M)$ and $f(y)=y_+$ the positive part, the so-called Kato inequality reads $\partial_{xx} f(y) \geq f'(y) \partial_{xx} y$.
Thus multiplying the latter equation by $\mathbf{1}_{n_\epsilon-n_M>0}$ and given \eqref{eq:segre} we obtain
\begin{align*}
\pa_t |n_\epsilon - n_M|_+ -\partial_{xx} |H(n_\epsilon)-H(n_M)|_+ \leq
({n}_\epsilon - n_M) \mathbf{1}_{x\leq  \zeta_{\epsilon}(t)}  G_1(P(n_\epsilon)) \mathbf{1}_{n_\epsilon-n_M>0} \\
+ ({n}_\epsilon - n_M) \mathbf{1}_{x\geq  \zeta_{\epsilon}(t)}  G_2(P(n_\epsilon)) \mathbf{1}_{n_\epsilon-n_M>0} \\
+ n_M (G_1(P(n_\epsilon)) - G_1(P(n_M))  + G_2(P(n_\epsilon)) - G_2(P(n_M))) \mathbf{1}_{n_\epsilon-n_M>0}.
\end{align*}
Since the function $P$ is increasing and $G_1$ and $G_2$ are decreasing (see \eqref{hypG}), we deduce that the last term is nonpositive.
Then, integrating on $(-L,L)$ and using $\partial_x n_\epsilon (\pm L,t) =0 $, we deduce
\begin{align*}
\frac{d}{dt} \int_{-L}^L |n_\epsilon - n_M|_+ \,dx \leq & \ \partial_{x} |H(n_\epsilon)-H(n_M)|_+(L,t) - \partial_{x} |H(n_\epsilon)-H(n_M)|_+(-L,t) \\
& + \int_{-L}^{ \zeta_{\epsilon}(t)} ({n}_\epsilon - n_M) \mathbf{1}_{n_\epsilon-n_M>0} G_1(P(n_\epsilon)) \,dx  \\
& + \int^{L}_{ \zeta_{\epsilon}(t)} ({n}_\epsilon - n_M) \mathbf{1}_{n_\epsilon-n_M>0} G_2(P(n_\epsilon)) \,dx \\
\leq &\  G_m \int_{-L}^L  |n_\epsilon - n_M|_+ \,dx.
\end{align*}

Then, integrating in time, we deduce
$$
\int_{-L}^L  |n_\epsilon - n_M|_+ \,dx \leq e^{G_m t} \int_{-L}^L  |n_\epsilon^{\mbox{\scriptsize ini}} - n_M|_+ \,dx = 0.
$$

{\bf $L^{\infty}$ bounds}.

From \eqref{hypini}, we have $p_{\epsilon}^{\mbox{\scriptsize ini}} \leq P_M$. Since the function $P$ is inscreasing, we have $ n_\epsilon^{\mbox{\scriptsize ini}} \leq n_M$. With the above comparison principle, we conclude that $n_\epsilon \leq n_M$.
We deduce easily with the non-negativity principle \eqref{lem:nonneg2} that $0 \leq p_\epsilon \leq P_M$, $ 0 \leq {n_1}_\epsilon\leq n_M$ and $ 0 \leq {n_2}_\epsilon \leq n_M$. 

{\bf Estimates from below}.

From above, we deduce that the pressure is bounded by $P_M$. Hence, using assumption \eqref{hypG} we deduce
$$
\pa_t n_\epsilon -\partial_{xx} H(n_\epsilon) =  {n_1}_\epsilon G_1(P(n_\epsilon))
+ {n_2}_\epsilon G_2(P(n_\epsilon)) \geq - n_\epsilon g_m.
$$
Let us introduce $n_m := A_0 e^{-g_m t}$. We deduce
$$
\pa_t (n_m-n_\epsilon) -\partial_{xx} (H(n_m) - H(n_\epsilon)) \leq -(n_m- n_\epsilon) g_m.
$$
As above, for the comparison principle, we may use the positive part and the Kato inequality to deduce 
$$
\pa_t |n_m-n_\epsilon|_+ -\partial_{xx} |H(n_m) - H(n_\epsilon)|_+ \leq - |n_m- n_\epsilon|_+ g_m.
$$
Integrating in space and in time as above, we deduce that $|n_m-n_\epsilon|_+=0$.

{\bf $L^{1}$ bounds of $n_{\epsilon}$, ${n_1}_{\epsilon}$, ${n_2}_{\epsilon}$ and ${p}_{\epsilon}$}.

Integrating \eqref{eq:nH} on $(-L,L)$ and using the nonnegativity of the densities from Lemma \ref{lem:nonneg2} as well as the Neumann boundary conditions, we deduce 
$$
\frac{d}{dt} \| n_{\epsilon} \|_{L^1(-L,L)} \leq G_m \|n_\epsilon\|_{L^1(-L,L)}.
$$
Integrating in time, we deduce
$$
\| n_{\epsilon} \|_{L^1(-L,L)} \leq e^{G_mt} \| n_{\epsilon}^{\mbox{\scriptsize ini}}\|_{L^1(-L,L)}.
$$
Since ${n_1}_{\epsilon} \geq 0$ and ${n_2}_{\epsilon} \geq 0$, we deduce the uniform bounds on $\| {n_1}_{\epsilon}\|_{L^1(-L,L)}$ and on $\| {n_2}_{\epsilon}\|_{L^1(-L,L)}$.

From the relation \eqref{eq:p2}, we deduce $p_{\epsilon}= n_{\epsilon} (\epsilon+p_{\epsilon})$. 
Moreover, the bound $p_{\epsilon} \leq P_M:=\max (P^1_M,P^2_M) $ implies
$$ 
\| p_{\epsilon}\|_{L^1(-L,L)} \leq (\epsilon+P_M)\int_{-L}^L |n_{\epsilon}|\,dx \leq Ce^{G_mt} \| n^{\mbox{\scriptsize ini}}_{\epsilon}\|_{L^1(-L,L)}.
$$

{\bf $L^{1}$ estimates on the $x$ derivatives}.

Recalling \eqref{eq:segre}, we can refomulate \eqref{eq:nn} by
\begin{equation}\label{eq:nn2}
\partial_t {n}_{\epsilon} - \partial_{xx} H({n}_{\epsilon}) = {n}_{\epsilon} G(p_{\epsilon},t,x) 
\end{equation}
with $G(p,t,x)= G_1(p) \mathbf{1}_{x \leq  \zeta_{\epsilon}(t) }+G_2(p)\mathbf{1}_{x \geq  \zeta_{\epsilon}(t) }$. The space derivative of this growth function is given by,
$$ 
\partial_{x} G(p,t,x) = (G_1(p)-G_2(p)) \delta_{ x= \zeta_{\epsilon}(t)} +G_1'(p) \partial_{x} p \mathbf{1}_{ x \leq  \zeta_{\epsilon}(t) } + G_2'(p) \partial_{x} p \mathbf{1}_{ x \geq  \zeta_{\epsilon}(t) }.
$$
We derive \eqref{eq:nn2} with respect to $x$,
\begin{align*}
\partial_t \partial_{x} n_{\epsilon} -\partial_{xx} (\partial_{x} H(n_{\epsilon})) =\ & \partial_x n_{\epsilon} G(p_{\epsilon},t,x) + n_{\epsilon} (G_1(p_{\epsilon})-G_2(p_{\epsilon})) \delta_{ x= \zeta_{\epsilon}(t)}  \\
&+ n_{\epsilon}(G_1'(p_{\epsilon})\mathbf{1}_{  x \leq  \zeta_{\epsilon}(t)} +G_2'(p_{\epsilon}) \mathbf{1}_{ x \geq  \zeta_{\epsilon}(t)}) \partial_{x} p_{\epsilon}.
\end{align*} 
We multiply by $ sign(\partial_{x} n_{\epsilon})=sign(\partial_{x} p_{\epsilon})$ and use the Kato inequality,
\begin{align*}
\partial_t |\partial_{x} n_{\epsilon}| -\partial_{xx} (|\partial_{x}H(n_{\epsilon})|) \leq \ & |\partial_x n_{\epsilon}| G(p_{\epsilon},t,x) \\
& + n_{\epsilon} (G_1(p_{\epsilon})-G_2(p_{\epsilon})) \delta_{ x= \zeta_{\epsilon}(t)} sign(\partial_{x_i} n_{\epsilon})\\
&+n_{\epsilon}(G_1'(p_{\epsilon})\mathbf{1}_{x \leq  \zeta_{\epsilon}(t)} +G_2'(p_{\epsilon}) \mathbf{1}_{ x \geq  \zeta_{\epsilon}(t)}) |\partial_{x} p_{\epsilon}|.
\end{align*}
We integrate in space on $(-L,L)$. Using the fact that $\max_{[0,P^1_M]} G_1'  \leq -\gamma <0$ and $\max_{[0,P^2_M]} G_2'  \leq -\gamma <0$ (see \eqref{hypG}) and that $\partial_x H(n_{\epsilon}) (\pm L,t)= H'(n_{\epsilon}) \partial_x n_{\epsilon} (\pm L,t) = 0$, 
\begin{align*}
\partial_t \int_{-L}^L |\partial_{x} n_{\epsilon}|\,dx \leq \ & G_m\int_{-L}^L|\partial_x n_{\epsilon}|\,dx - \gamma \int_{-L}^L n_{\epsilon} |\partial_{x} p_{\epsilon}|\,dx \\
&+ n_{\epsilon}(t, \zeta_{\epsilon}(t)) |G_1(p_{\epsilon}(t, \zeta_{\epsilon}(t))-G_2(p_{\epsilon}(t, \zeta_{\epsilon}(t))| .
\end{align*}
Using Gronwall's lemma and the uniform bound on $n_\epsilon$ and $G_1$ and $G_2$ (see \eqref{hypG}), we deduce that, for all $t>0$,
\begin{equation}\label{eq:estimdxn2} 
\| \partial_{x} n_{\epsilon}(t) \|_{L^1(-L,L)} + \gamma \int_0^t \int_{-L}^L n_{\epsilon} |\partial_{x} p_{\epsilon}|\,dxds  \leq C e^{G_mt}  \left(\| \partial_{x} n_{\epsilon}^{\mbox{\scriptsize ini}} \|_{L^1(-L,L)} + 1\right) .
\end{equation}
This conclude the proof for the estimate on $\partial_x n_{\epsilon}$. Then,
$$
\| \partial_{x} p_{\epsilon}\|_{L^1(-L,L)} = \int_{-L}^L |\partial_{x} p_{\epsilon}|\,dx = \int_{-L}^L\frac{\epsilon}{(1-n_{\epsilon})^2}| \partial_{x} n_{\epsilon}|\,dx.
$$
We split the latter integral in two: either $n_{\epsilon}\leq 1/2$ and then $\frac{\epsilon}{(1-n_{\epsilon})^2} \leq C$; either $n_{\epsilon}\geq 1/2$,
\begin{align*} 
  \| \partial_{x} p_{\epsilon}\|_{L^1(-L,L)}  &\leq  C \int_{n_{\epsilon}\leq 1/2} |\partial_{x} n_{\epsilon}|\,dx + \int_{n_{\epsilon}\geq 1/2} |\partial_{x} p_{\epsilon}|\,dx \\
  &\leq  C \int_{n_{\epsilon}\leq 1/2} |\partial_{x} n_{\epsilon}|\,dx + 2 \int_{n_{\epsilon}\geq 1/2} \frac{1}{2}|\partial_{x} p_{\epsilon}|\,dx \\
   &\leq  C \|\partial_{x} n_{\epsilon}(t)\|_{L^1(-L,L)} + 2 \int_{n_{\epsilon}\geq 1/2} n_{\epsilon}|\partial_{x} p_{\epsilon}|\,dx.
  \end{align*} 
 Then, we integrate in time and we deduce using \eqref{eq:estimdxn2}
 \begin{equation*}
  \| \partial_{x} p_{\epsilon} \|_{L^1(Q_T)} \leq C'e^{G_mT} (\| \partial_{x} n_{\epsilon} \|_{L^1(-L,L)}\textcolor{magenta}{+1)}.
 \end{equation*}
Hence we have an uniform bound on $\partial_x p_\epsilon$ in $L^1(Q_T)$.
 To recover the estimate on $ \partial_{x} {n_1}_{\epsilon}$ and $ \partial_{x} {n_2}_{\epsilon}$ we deduce from \eqref{eq:segre},
$$ \partial_{x} {n_1}_{\epsilon}= \partial_{x} {n}_{\epsilon} \mathbf{1}_{ x \leq  \zeta_{\epsilon}(t)}+n_{\epsilon} \delta_{x =  \zeta_{\epsilon}(t)},$$
$$ \partial_{x} {n_2}_{\epsilon}= \partial_{x} {n}_{\epsilon} \mathbf{1}_{x \leq  \zeta_{\epsilon}(t)}-n_{\epsilon} \delta_{ x =  \zeta_{\epsilon}(t)}.$$
So
  $$ \| \partial_{x} {n_1}_{\epsilon} \|_{L^1(-L,L)} = \int_{x \leq  \zeta_{\epsilon}(t)} \partial_{x_i} {n_1}_{\epsilon} ~dx+ {n}_{\epsilon}(t, \zeta_{\epsilon}(t))    \leq \| \partial_{x} n_{\epsilon} \|_{L^1(-L,L)} + \| {n_1}_{\epsilon}\|_{\infty},$$
  and
$$ \| \partial_{x} {n_2}_{\epsilon} \|_{L^1(-L,L)} = \int_{x \geq  \zeta_{\epsilon}(t)} \partial_{x_i} {n_2}_{\epsilon}  ~dx- {n}_{\epsilon}(t, \zeta_{\epsilon}(t))   \leq \| \partial_{x} n_{\epsilon} \|_{L^1(-L,L)} + \| {n_2}_{\epsilon}\|_{\infty}$$
This concludes the proof.
\end{proof}

\subsubsection{ \texorpdfstring{$L^2$}{L2} estimate for \texorpdfstring{$\partial_x p$}{dxp}}

\begin{lem}[$L^2$ estimate for $\partial_x p$]\label{lem:L2dp2}
Let us assume that \eqref{hypG} and \eqref{hypini} hold. 
Let $({n_1}_\epsilon,{n_2}_\epsilon,p_\epsilon)$ be a solution to \eqref{eq:n1}--\eqref{eq:n}.
Then, for all $T>0$ we have a uniform bound on $\partial_x p_\epsilon$ in $L^2(Q_T)$.
\end{lem}
\begin{proof}

For a given function $\psi$ we have, multiplying \eqref{eq:n} by $\psi(n_\epsilon)$,
$$
\partial_t n_{\epsilon} \psi(n_\epsilon) -\partial_x(n_{\epsilon} \partial_x p_{\epsilon})\psi(n_{\epsilon})= ({n_1}_{\epsilon}  G_1(p_{\epsilon} ) +{n_2}_{\epsilon}  G_2(p_{\epsilon} ))\psi(n_{\epsilon}).
$$
Integrating on $(-L,L)$, we have 
$$ 
\frac{d}{dt} \int_{-L}^L \Psi(n_{\epsilon})\,dx + \int_{-L}^L n_{\epsilon} \partial_x n_{\epsilon}\cdot \partial_x p_{\epsilon}\psi'(n_{\epsilon})\,dx = \int_{-L}^L ({n_1}_{\epsilon}  G_1(p_{\epsilon} ) +{n_2}_{\epsilon}  G_2(p_{\epsilon} ))\psi(n_{\epsilon} )\,dx,
$$
where  $\Psi$ is an antiderivative of $\psi$. We choose $\psi(n)= \epsilon (\ln(n)-\ln(1-n)+\frac{1}{1-n})$ so that $n_{\epsilon} \psi'(n_{\epsilon})=P'(n_{\epsilon})$. Inserting the expression of $\psi$, we get
\begin{align*}
&\frac{d}{dt} \int_{-L}^L \epsilon n_{\epsilon} \ln\Big(\frac{n_{\epsilon}}{1-n_{\epsilon}}\Big)\,dx 
+ \int_{-L}^L |\partial_x p_{\epsilon}|^2 dx  \leq G_{m} \int_{-L}^L \epsilon n_{\epsilon}\left|\ln(n_{\epsilon})-\ln(1-n_{\epsilon})+\frac{1}{1-n_{\epsilon}}\right|\,dx.
\end{align*}
After integrating in time and using the expression of the pressure \eqref{eq:p2}, we have
\begin{align*}
\int_{-L}^L \epsilon n_{\epsilon} \ln\Big(\frac{p_{\epsilon}}{\epsilon}\Big)\,dx 
& -\int_{-L}^L \epsilon n_{\epsilon}^{\mbox{\scriptsize ini}} \ln\left(\frac{n_{\epsilon}^{\mbox{\scriptsize ini}}}{1-n_{\epsilon}^{\mbox{\scriptsize ini}}}\right)\,dx + \int_0^T \int_{-L}^L |\partial_x p_{\epsilon}|^2\,dxdt   \\
&\leq G_{m} \int_0^T \int_{-L}^L \left(\epsilon n_{\epsilon} \Big|\ln\Big(\frac{p_{\epsilon}}{\epsilon}\Big)\Big|+p_{\epsilon}\right) \,dx.
\end{align*}
Then, to prove that $\partial_x p_\epsilon \in L^2(Q_T)$, we are left to find a uniform bound on $ \int_{-L}^L \epsilon n_{\epsilon} |\ln(\frac{p_{\epsilon}}{\epsilon})|dx$. Using the expression of $p_\epsilon$ in \eqref{eq:p2}, we have
\begin{align*}
\int_{-L}^L \epsilon n_{\epsilon} |\ln\big(\frac{p_{\epsilon}}{\epsilon}\big)|\,dx & \leq \int_{-L}^L \epsilon n_{\epsilon} |\ln p_{\epsilon}| \,dx  + \epsilon \ln(\epsilon) \int_{-L}^Ln_{\epsilon}\,dx \\
 & \leq  \int_{-L}^L (1-n_{\epsilon})p_{\epsilon} |\ln p_{\epsilon}|\,dx + \epsilon \ln(\epsilon) \int_{-L}^L n_{\epsilon}\,dx
\end{align*}
Since $n_\epsilon$ is bounded in $L^1$, the second term of the right hand side is uniformly bounded with respect to $\epsilon$.
Moreover given that $0\leq p_\epsilon \leq P_M$ and $x\mapsto x|\ln x|$ is uniformly bounded on $[0,P_M]$, we get
$$  \int_{-L}^L (1-n_{\epsilon})p_{\epsilon} |\ln(p_{\epsilon})|\,dx \leq C \int_{-L}^L \mathbf{1}_{p_{\epsilon}>0} \,dx \leq 2LC.$$
This concludes the proof.

\end{proof}

  \subsection{Proof of theorem 1}
  
    \subsubsection{Convergence}
  
  In the last paragraph we have found a priori estimates for the densities and their space derivatives. To use a compactness argument, we need to obtain estimates on the time derivative. To do so, we are going to use the Aubin Lions theorem \cite{simon}.

According to Lemma \ref{lem:L2dp2}, ${n_1}_{\epsilon}\partial_x p_\epsilon$ and ${n_2}_{\epsilon}\partial_x p_\epsilon$ are in $L^2(Q_T)$. Moreover thanks to Lemma \ref{lem:estim}, we have that ${n_1}_{\epsilon} G_1(p_{\epsilon})$ and ${n_2}_{\epsilon} G_2(p_{\epsilon})$ are uniformly bounded in $L^\infty([0,T];L^1\cap L^\infty(-L,L))$, so $\partial_t {n_1}_{\epsilon}$ and $\partial_t {n_2}_{\epsilon}$ are uniformly bounded in $ L^2([0,T],W^{-1,2}(-L,L))$. We also have ${n_1}_{\epsilon}$ and ${n_2}_{\epsilon}$ bounded in $L^1([0,T],W^{1,1}(-L,L))$. Since we are working in one dimension, we have the following embeddings
$$
W^{1,1}(-L,L) \subset L^1(-L,L) \subset W^{-1,2}(-L,L).
$$
 The Aubin Lions theorem implies that $\{u \in L^1([0,T],W_{loc}^{1,1}(-L,L)); \dot{u} \in L^2([0,T],W^{-1,2}(-L,L))  \}$ is compactly embedded in $L^1([0,T],L^1(-L,L))$.
So we can extract strongly converging subsequences ${n_1}_{\epsilon}$ and  ${n_2}_{\epsilon}$ in $L^1(Q_T)$. The convergence of the pressure follows from the same kind of computation. 
 
 \subsubsection{Limit model}
 
From the above results, up to extraction of subsequences, $(n_{1\epsilon})_\epsilon$, $(n_{2\epsilon})_\epsilon$, and $(p_\epsilon)_\epsilon$ converge strongly in $L^1(Q_T)$ and a.e. towards some limits denoted $n_{10}$, $n_{20}$, and $p_0$, respectively.
Moreover, due to the uniform estimate on $(\partial_x p_\epsilon)_\epsilon$ in $L^2(Q_T)$ from Lemma \ref{lem:L2dp2}, we may extract a subsequence, still denoted $(\partial_x p_\epsilon)_\epsilon$, which converges weakly in $L^2(Q_T)$ towards $\partial_x p_0$.
Passing to the limit in the uniform estimates of Lemma \ref{lem:estim} gives \eqref{unif0} and ${n_1}_0, {n_2}_0, n_0, p_0$ belongs to $BV(Q_T)$.

Then, we recall that 
$$
\partial_t n_\epsilon - \partial_{xx} (p_\epsilon - \epsilon \ln(p_\epsilon + \epsilon)) = {n_1}_\epsilon G_1(p_\epsilon)+{n_2}_\epsilon G_2(p_\epsilon).
$$
From the uniform bounds on $p_\epsilon$, we get,
$$
\epsilon \ln \epsilon \leq \epsilon \ln(p_\epsilon + \epsilon) \leq \epsilon \ln(P_M + \epsilon).
$$
Thus, the term in the Laplacian converges strongly to $p_0$. Then, thanks to the strong convergence of $n_\epsilon$
and $p_\epsilon$, we deduce that in the sense of distributions
$$
\partial_t n_0 - \partial_{xx} p_0 = {n_1}_0 G_1(p_0)+{n_2}_0 G_2(p_0).
$$

Moreover, let $\phi\in W^{1,\alpha}(Q_T)$ with $\phi(T,x)=0$ ($\alpha>2$) be a test function.
We multiply equation \eqref{eq:n1} by $\phi$ and integrate using the Neumann boundary conditions, we get
\begin{align*}
-\int_0^T\int_{-L}^L {n_1}_\epsilon \partial_t \phi \,dtdx
  - \int_{-L}^L {n_1}_\epsilon^{ini}(x) \phi(0,x)\,dx + \int_0^T\int_{-L}^L {n_1}_\epsilon \partial_x p_\epsilon \partial_x \phi \,dxdt & \\
= \int_0^T\int_{-L}^L {n_1}_\epsilon G_1(p_\epsilon) \phi \,dxdt.&
\end{align*}
Due to the strong convergence of ${n_1}_\epsilon$ and $p_\epsilon$, we can pass easily to the limit $\epsilon\to 0$ into the first term of the left hand side and into the term in the right hand side.
For the second term, we use the assumptions on the initial data to pass into the limit.
For the third term, we can pass to the limit in a product of a weak-strong convergence from standard arguments, then we arrive at
\begin{align*}
-\int_0^T\int_{-L}^L {n_1}_0 \partial_t \phi \,dtdx
  - \int_{-L}^L n_1^{ini}(x) \phi(0,x)\,dx + \int_0^T\int_{-L}^L {n_1}_0 \partial_x p_0 \partial_x \phi \,dxdt & \\
= \int_0^T\int_{-L}^L {n_1}_0 G_1(p_0) \phi \,dxdt,&
\end{align*}
for any test function $\phi\in W^{1,\alpha}(Q_T)$.
Then we obtain the weak formulation of \eqref{limitn1} with Neumann boundary conditions on $p_0$.
We proceed by the same token to recover \eqref{limitn2}.

Passing into the limit in the relation $(1-n_\epsilon) p_\epsilon = \epsilon n_\epsilon$ implies
$$ (1-n_0) p_0 =0.$$
We can also pass to the limit for the segregation and deduce
${n_1}_0 {n_2}_0=0$.
To conclude the proof of Theorem \ref{TH1}, we are left to establish the relation \eqref{relcompbis}.

\subsection{Complementary relation}

In this section we want to pass to the limit in the equation for the pressure \eqref{eq:pp}.
However, this task can not be performed easily since we only have uniform estimates on the gradient of $n$ and $p$, whereas we need strong convergence of the gradient to pass to the limit in \eqref{eq:pp}.
Then we propose to work on the time antiderivative. Let us denote $q_\epsilon = p_\epsilon-\epsilon\ln(p_\epsilon+\epsilon)$. 
Then, we have proved above that $q_\epsilon \to p_0$ strongly as $\epsilon\to 0$, and
\begin{equation}\label{eq:nq}
\partial_t n_\epsilon - \partial_{xx} q_\epsilon = n_{1\epsilon}G_1(p_\epsilon) + n_{2\epsilon} G_2(p_\epsilon).
\end{equation}
Let us introduce $Q_\epsilon$ a time antiderivative of $q_\epsilon$, $Q_\epsilon(t,x):= \int_0^t q_\epsilon(s,x)\,ds$.
From the strong convergence of $q_\epsilon$, we deduce that $Q_\epsilon \to P_0:=\int_0^t p_0(s,x)\,ds$ as $\epsilon\to 0$.
By a simple time integration of \eqref{eq:nq}, we have
\begin{equation}\label{eq:nqint}
\partial_{xx} Q_\epsilon = n_\epsilon - n^{\mbox{\scriptsize ini}}_\epsilon - 
\int_0^t (n_{1\epsilon}G_1(p_\epsilon) + n_{2\epsilon} G_2(p_\epsilon))\,ds.
\end{equation}
From Lemma \ref{lem:estim}, we deduce that $\partial_{xx} Q_\epsilon$ is uniformly bounded in 
$L^1\cap L^\infty([0,T]\times(-L,L))$. Moreover, using the relation $q_\epsilon = p_\epsilon-\epsilon\ln(p_\epsilon+\epsilon)$, we get
$$
\partial_t\partial_x Q_\epsilon = \partial_x q_\epsilon = \frac{p_\epsilon}{p_\epsilon+\epsilon}\partial_x p_\epsilon.
$$
From the uniform bound on $\partial_x p_\epsilon$ in $L^2(Q_T)$ in Lemma \ref{lem:L2dp2}, 
we deduce that the sequence $(\partial_t\partial_x Q_\epsilon)_\epsilon$ is uniformly bounded
in $L^2([0,T]\times(-L,L))$.
Thus we have obtained that the sequence $(\partial_xQ_\epsilon)_\epsilon$ is uniformly bounded
in $H^1([0,T]\times(-L,L))$. We deduce from the compact embedding of $H^1(Q_T)$ into $L^2(Q_T)$
that we can extract a subsequence, still denoted $(\partial_xQ_\epsilon)_\epsilon$, converging
strongly in $L^2(Q_T)$ and weakly in $H^1(Q_T)$ towards a limit denoted $\overline{\partial Q}$. Since
$Q_\epsilon\to  P_0$ as $\epsilon\to 0$, we deduce $\overline{\partial Q}=\partial_x P_0$.

Thus, we can pass to the limit $\epsilon\to 0$ into the equation \eqref{eq:nqint}. We obtain
$$
n_0 - n^{\mbox{\scriptsize ini}}_0 - \partial_{xx} P_0 = \int_0^t (n_{10}G_1(p_0) + n_{20} G_2(p_0))\,ds.
$$
Multiplying by $p_0$ and using the relation $p_0n_0=p_0$, we deduce the complementary relation  \eqref{relcompbis}.
This concludes the proof of Theorem \ref{TH1}.

%%%%%%%%%%%%%%%%%%%%%%%%%%%%%%%%%%%%%%%%%%%%%%%%
\subsection{Uniqueness of solutions}
\label{sec:uniq}

In this section, we focus on the uniqueness of solutions to the limiting problem \eqref{limitn0}--\eqref{segre}.
We first observe that from \eqref{limitn0} and \eqref{segre}, we have 
\begin{equation}\label{eqn0bis}
\pa_t n_0 -\pa_{xx} (n_0 p_0) = n_{10} G_1(p_0) + n_{20} G_2(p_0), \qquad \mbox{ in } \ \mathcal{D}'(Q_T).
\end{equation}
Since we have the segregation property given by \eqref{segre}, we deduce that the support of $n_{10}$ and of $n_{20}$ are disjoints. Then, by taking test functions with support included in the support of $n_{10}$ or of $n_{20}$ in the weak formulation of \eqref{eqn0bis}, we deduce that 
\begin{align}
\pa_t n_{10} -\pa_{xx} (n_{10} p_0) = n_{10} G_1(p_0), \qquad &\mbox{ in } \ \mathcal{D}'(Q_T), \label{limitn1bis} \\
\pa_t n_{20} -\pa_{xx} (n_{20} p_0) = n_{20} G_2(p_0), \qquad &\mbox{ in } \ \mathcal{D}'(Q_T). \label{limitn2bis}
\end{align}
We are going to prove that system \eqref{limitn1bis}--\eqref{limitn2bis} complemented with the segregation property \eqref{segre} and the relation \eqref{n0p02} admits an unique solution. More precisely our result reads:
\begin{prop}\label{TH2}
Let us assume that assumptions \eqref{hypG} on $G_i$, $i=1,2$ holds.
There exists a unique solution $(n_{10}, n_{20}, p_0)$ to the problem \eqref{limitn1bis}-\eqref{limitn2bis}-\eqref{n0p02}-\eqref{segre} with $0\leq n_{i0}\leq 1$ for $i=1,2$.
\end{prop} 

\begin{proof}
We follow the idea developped in \cite{PQV} and adapt the Hilbert's duality method.
Consider two solutions $(n_{10}, n_{20}, p_0)$ and $(\widetilde{n_{10}}, \widetilde{n_{20}}, \widetilde{p_0})$ of the system \eqref{limitn1bis}-\eqref{limitn2bis}-\eqref{n0p02}-\eqref{segre}. 
Making the difference and denoting $q_i = n_{i0} p_0$ and $\widetilde{q_i}=\widetilde{n_{i0}} \widetilde{p_0}$, for $i=1,2$, we have
\begin{align*}
\pa_t (n_{10}-\widetilde{n_{10}}) -\pa_{xx} (q_1 - \widetilde{q_1}) = n_{10} G_1(p_0) - \widetilde{n_{10}} G_1(\widetilde{p_0}), \qquad &\mbox{ in } \ \mathcal{D}'(Q_T), \\
\pa_t (n_{20}-\widetilde{n_{20}}) -\pa_{xx} (q_2 - \widetilde{q_2}) = n_{20} G_2(p_0) - \widetilde{n_{20}} G_2(\widetilde{p_0}), \qquad &\mbox{ in } \ \mathcal{D}'(Q_T).
\end{align*}
We first observe that on the set $\{n_{10}>0\}\cap\{p_0>0\}$, we have $q_1 = p_0$ from \eqref{n0p02}. Hence we have $n_{10} G_1(p_0) = n_{10} G_1(q_1)$. The same observation holds for the other terms in the right hand side of these latter equations.
For any suitable test functions $\psi_1$ and $\psi_2$, we have, for $i=1,2$,
\begin{equation}\label{eq:uniq1}
\iint_{Q_T} \Big[(n_{i0}-\widetilde{n_{i0}})\pa_t \psi_i + (q_i - \widetilde{q_i}) \pa_{xx} \psi_i + (n_{i0} G_i(q_i) - \widetilde{n_{i0}} G_i(\widetilde{q_i}))\psi_i\Big] \,dxdt = 0.
\end{equation}
This can be rewritten as, for $i=1,2$,
\begin{equation}\label{eq:uniq2}
\iint_{Q_T} (n_{i0}-\widetilde{n_{i0}}+q_i - \widetilde{q_i})\Big(A_i\pa_t \psi_i + B_i \pa_{xx} \psi_i + A_i G_i(q_i) \psi_i - C_i B_i\psi_i\Big) \,dxdt = 0,
\end{equation}
where
$$
A_i = \frac{n_{i0}-\widetilde{n_{i0}}}{n_{i0}-\widetilde{n_{i0}}+q_i - \widetilde{q_i}}, \qquad
B_i = \frac{q_i - \widetilde{q_i}}{n_{i0}-\widetilde{n_{i0}}+q_i - \widetilde{q_i}}, \qquad
C_i = - \widetilde{n_{i0}} \frac{G_i(q_i)-G_i(\widetilde{q_i})}{q_i-\widetilde{q_i}},
$$
and we define $A_i=0$ as soon as $n_{i0}=\widetilde{n_{i0}}$ and 
$B_i=0$ as soon as $q_i = \widetilde{q_i}$, whatever is the value of their denominators. It is shown in Lemma \ref{lemABC} below that, for $i=1,2$, we have 
$0\leq A_i\leq 1$, $0\leq B_i \leq 1$, $0\leq C_i \leq \gamma$.

The idea of the Hilbert's duality method consists in solving the \textit{dual problem}, which is defined here by, for any smooth function $\Phi_i$, $i=1,2$,
\begin{equation}\label{dualpb}
\left\{\begin{array}{l}
A_i\pa_t \psi_i + B_i \pa_{xx} \psi_i + A_i G_i(q_i) \psi_i - C_i B_i\psi_i = A_i \Phi_i, \quad \mbox{ in } \ Q_T,  \\[1mm]
\partial \psi_i(\pm L) = 0 \quad \mbox{ in } \  (0,T), \qquad
\psi_i(\cdot,T) = 0 \quad \mbox{ in } (-L,L).
\end{array}\right.
\end{equation}
If such a system admits a smooth solution, then, by choosing $\psi_i$ as a test function in \eqref{eq:uniq2}, we get
$$
\iint_{Q_T} (n_{i0}-\widetilde{n_{i0}}+q_i - \widetilde{q_i}) A_i \Phi_i \,dxdt = 0.
$$
From the expression of $A_i$, we deduce
$$
\iint_{Q_T} (n_{i0}-\widetilde{n_{i0}}) \Phi_i \,dxdt = 0,
$$
for any smooth function $\Phi_i$, $i=1,2$. 
It is obvious to deduce the uniqueness for the density. Uniqueness for the pressure will follow from \eqref{eq:uniq1}.

However, the dual problem \eqref{dualpb} is not uniformly parabolic and its coefficients are not smooth. Then, in order to make this step rigorous, a regularization procedure is required. It can be done exactly as in \cite[p 109-110]{PQV}. For the sake of completeness of this paper, this regularizing procedure is recalled in Appendix \ref{appen}.
\end{proof}

\begin{lem}\label{lemABC}
Under assumptions \eqref{hypG}, 
we have $0\leq A_i\leq 1$, $0\leq B_i \leq 1$, $0\leq C_i \leq \gamma$, for $i=1,2$.
\end{lem}
\begin{proof}
We observe that, for $i=1,2$, $n_{i0}>\widetilde{n_{i0}}$ implies $q_i \geq \widetilde{q_i}$. Indeed, either $\widetilde{n_{i0}}=0$ and then $\widetilde{q_{i}}=0\leq q_i$, or $0<\widetilde{n_{i0}}<1$ and then from the segregation property \eqref{segre} we have $\widetilde{n_0}=\widetilde{n_{i0}}$ and from the relation $(1-\widetilde{n_0})\widetilde{p_0}=0$ we deduce that $\widetilde{p_0}=0$, thus $\widetilde{q_{i}}=0\leq q_i$.
Similarly, for $i=1,2$, $\widetilde{n_{i0}}>n_{i0}$ implies $\widetilde{q_i} \geq q_i$. By setting $A_i=0$ whenever $\widetilde{n_{i0}} = n_{i0}$, we conclude that $0\leq A_i \leq 1$.

By the same token, we show that, for $i=1,2$, $q_i \geq \widetilde{q_i}$ implies $n_{i0} \geq \widetilde{n_{i0}}$. Indeed, from $q_i = n_{i0} p_0 >0$, we deduce that $n_{i0}>0$ which implies $n_0=n_{i0}$, and then $p_0>0$ implies from \eqref{n0p02} that $n_{i0}=1\geq \widetilde{n_{i0}}$. Hence, $0\leq B_i\leq 1$.

Finally, the bound on $C_i$ is a direct consequence of the fact that $G_i$ is nonincreasing and Lipschitz (see \eqref{hypG}) and that $0\leq \widetilde{n_{i0}} \leq 1$.
\end{proof}

\section{Numerical simulations}
\label{sec:numeric}

\subsection{Numerical scheme}

The numerical simulations are performed using a finite volume method similar as the one proposed in \cite{CCH,CDHV}. The scheme used for the conservative part is a classical explicit upwind scheme. To facilitate the reading of this paper, we recall here the scheme used. We divide the computational domain into finite-volume cells $C_j =[x_{j-1/2} ,x_{j+1/2} ] $ of uniform size $\Delta x$ with $x_j =j\Delta x$, $j\in \{1,...,M_x\} $,  and $x_{j}= \frac{x_{j-1/2} +x_{j+1/2} }{2}$ so that
$$ 
-L= x_{1/2}  < x_{3/2} < ... <  x_{j-1/2} < x_{j+1/2} < ... < x_{M_x-1/2}  < x_{M_x+1/2} = L,
$$
and define the cell average of functions $n_1(t,x)$ and $n_2(t,x)$ on the cell $C_j$ by
$$ 
\bar{n}_{\beta_j}(t)= \frac{1}{\Delta x} \int_{C_j}n_\beta(t,x)\,dx,\quad  \beta \in \{1,2\}.
$$
The scheme is obtained by integrating system \eqref{eq:n1}-\eqref{eq:n2} over $C_j$ and is given by
\begin{equation} \label{eq:scheme}
{\bar{n}}_{\beta_j}^{k+1}= - \frac{F^k_{\beta,j+1/2}-F^k_{\beta,j-1/2}}{\Delta x}+{\bar{n}}_{\beta_j}^{k+1} G_\beta(p_{j}^{k}) \quad \mbox{for } \beta =1,2 ,
\end{equation}
where $F^k_{\beta,j+1/2}$ are numerical fluxes approximating 
$-n^k_{\beta} u^k_{\beta}:=-n^k_{\beta} \partial_x(p^k_{\beta})$ and defined by:
	\begin{equation*}
	 F^k_{\beta,j+1/2} =(u^k_{{\beta}_{j+1/2}})^+\bar{n}^k_{\beta_j} +(u^k_{{\beta}_{j+1/2}})^- {\bar{n}}^k_{\beta_{j+1}},\quad \beta \in \{1,2\},
	 \end{equation*}
where
\begin{equation*}
	 {u_\beta}^k_{j+1/2} = \left\{
      \begin{aligned}
        &- \frac{p^k_{{j+1}}-p^k_{j}}{\Delta x}, & \quad \forall j \in \{2,...,M_x-1\},\\
        &0, & \quad \mbox{ otherwise },
        \end{aligned}
    \right.
	\end{equation*}
with the discretized pressure 
$$
p^k_{j} = \frac{\epsilon n^k_j}{1-n^k_j}, \quad 
n_j^k = \bar{n}^k_{1_j}+\bar{n}^k_{2_j}.
$$
We use the usual notation $(u)^+=\max(u,0)$ and $(u)^-=\min(u,0)$ for the positive part and, respectively, the negative part of $u$.
Neumann boundary conditions are also implemented at the boundaries of the computational model.

In order to illustrate the time dynamics for the model, we plot in Fig \ref{fig:1} the densities computed thanks to the above scheme for $\epsilon = 1$ at different times~: (a) $t=0$, (b) $t=0.1$,  (c) $t=0.3$,  (d) $t=0.6$, (e) $t=1$ and (f) $t=2$.
For this numerical simulation, the densities are initialized by
\begin{equation}\label{inifig}
 n_1^{\mbox{\scriptsize ini}}(x)=0.98 ~\mathbf{1}_{[-L;0.25]} (x) \quad \mbox{ and } \quad n_2^{\mbox{\scriptsize ini}}(x)= 0.98~ \mathbf{1}_{[0.25;L]} (x), 
\end{equation}
with $L=5$,
and the growth rates are defined by
\begin{equation}\label{GRfig}
 G_1(p)=10(1-p/2)  \quad \mbox{ and } \quad G_2(p)=10(1-p).
\end{equation}
We recall that we have defined the parameters $P_M^1$ and $P_M^2$ as the values of the pressure for which the growth functions vanish (see \eqref{hypG}). In this case their numerical values are given by $P_M^1=2$ and $P_M^2=1$.
Then, we define
\begin{equation}\label{P_M}
{N_M^1}_{\epsilon}=p^{-1}(P_M^1)=\frac{P_M^1}{\epsilon+P_M^1} \quad \mbox{ and } \quad {N_M^2}_{\epsilon}=p^{-1}(P_M^2)=\frac{P_M^2}{\epsilon+P_M^2}.
\end{equation}
Since the growth functions are different, clearly ${N_M^2}_\epsilon<{N_M^1}_\epsilon$.
\begin{figure}[!ht]
   \centering
\includegraphics[width=0.8\linewidth]{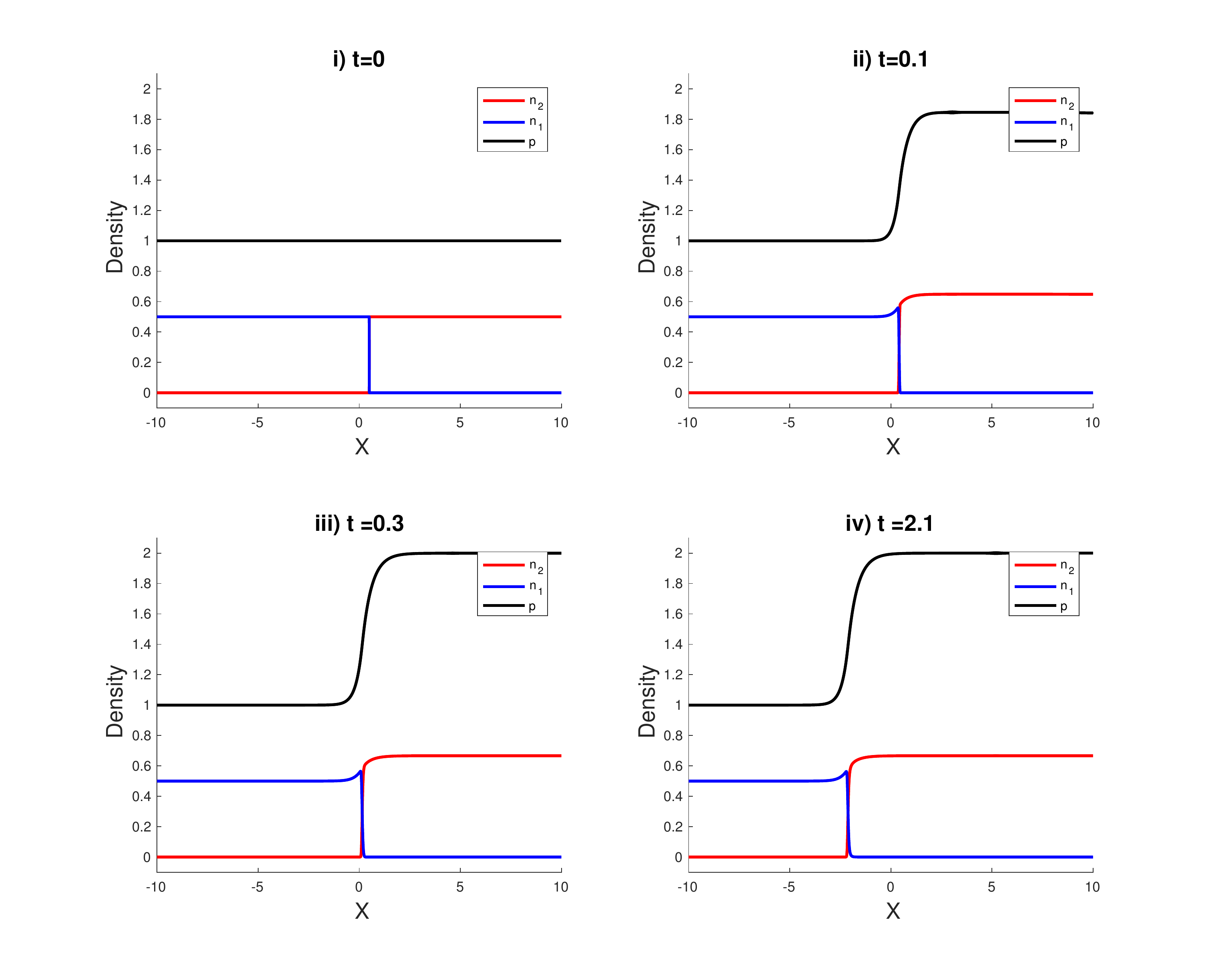}
\caption{Densities $n_1$ (red), $n_2$ (blue) and pressure $p$ as functions of position $x$ at different times: a) $t=0$, (b) $t=0.1$,  (c) $t=0.3$,  (d) $t=0.6$, (e) $t=1$ and (f) $t=2$; in the case $\epsilon=1$ with the initial densities and growth rate defined by \eqref{inifig}-\eqref{GRfig}.}
\label{fig:1}
\end{figure}

In Fig \ref{fig:1} the red and blue species are initially segregated and equal to 0.5. At first the dynamics is driven by the growth term, so the two species grow and reach their respective maximal packing values ${N_M^1}_{\epsilon}$ and ${N_M^2}_{\epsilon}$.
Once this value is reached ($t=1, 2$ on both panel (ii), (iii) and (iv)), we observe two phenomena. First a bump is created on the left side of the interface, in the domain of $n_2$. This bump help the total densities to stay continuous, as it joins the two maximal densities. It also means that, at the interface, the pressure is going to be higher than the limit pressure $P_M^2$. Then the derivative of the pressure at the interface is positive, which induces a motion of the interface representing the fact that the red species $n_1$ pushes the blue species $n_2$. This motion of the interface is the second phenomenon which is observed.

\subsection{Influence of the parameter \texorpdfstring{$\epsilon$}{e}}
\label{sec:numeps}

In order to illustrate our main result on the limit $\epsilon\to 0$, we show, in this section, some numerical simulations of the model \eqref{eq:n1}-\eqref{eq:n2} when $\epsilon$ goes to $0$. 
We also compare with the analytical solution of the limiting Hele-Shaw free boundary model. To perform these simulations we use the numerical scheme \eqref{eq:scheme} complemented with the initial condition \eqref{inifig} and the growth function \eqref{GRfig}. For the limiting model, we use the initial conditions 
\begin{equation*}
 n_1^{\mbox{\scriptsize ini}}(x)= \mathbf{1}_{[-L;0.25]} (x) \quad \mbox{ and } \quad n_2^{\mbox{\scriptsize ini}}(x)= \mathbf{1}_{[0.25;L]} (x), 
\end{equation*}
and the growth function \eqref{GRfig}. The analytical expressions of the solution to the limiting Hele-Shaw system is computed in \cite{CDHV}.

Fig \ref{fig:2} displays the time dynamics of the densities for different values of $\epsilon$:  (a) $\epsilon = 1$, (b) $\epsilon = 0.1$, (c) $\epsilon = 0.01$, and (d) $\epsilon = 0.001$, along with solution to the Hele-Shaw system (e). For all simulations, the densities are plotted at times $t=0.5$, $t=1$ and $t=1.5$.

  \begin{figure}
   \centering
   \begin{subfigure}[ $\epsilon=1$]{
\includegraphics[width=0.7\linewidth]{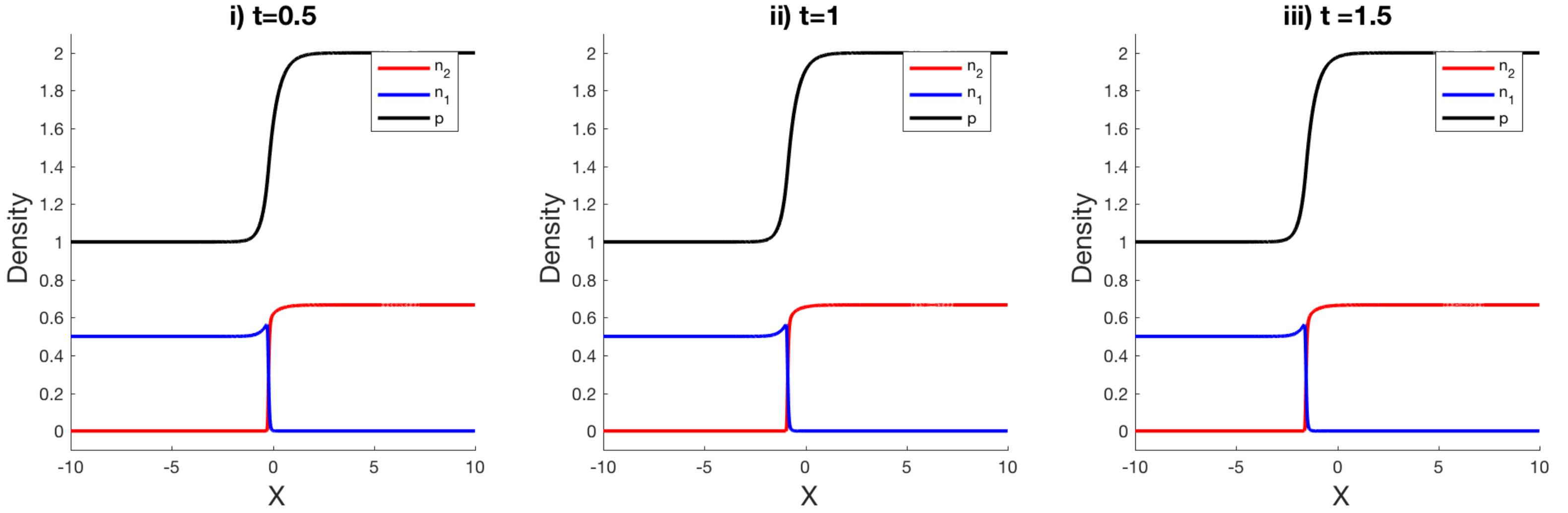} 
}
\end{subfigure}
\hfill
 \begin{subfigure}[$\epsilon=0.1$]{
\includegraphics[width=0.7\linewidth]{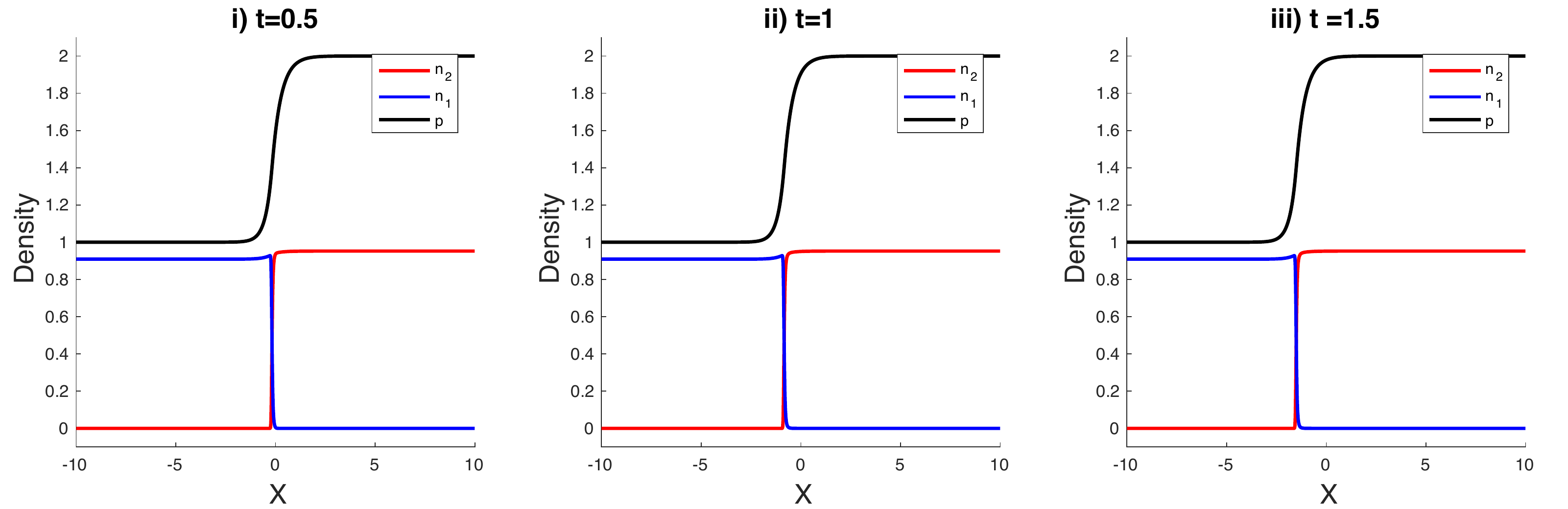}
}
\end{subfigure}
\hfill
   \begin{subfigure}[$\epsilon=0.01$]{
\includegraphics[width=0.7\linewidth]{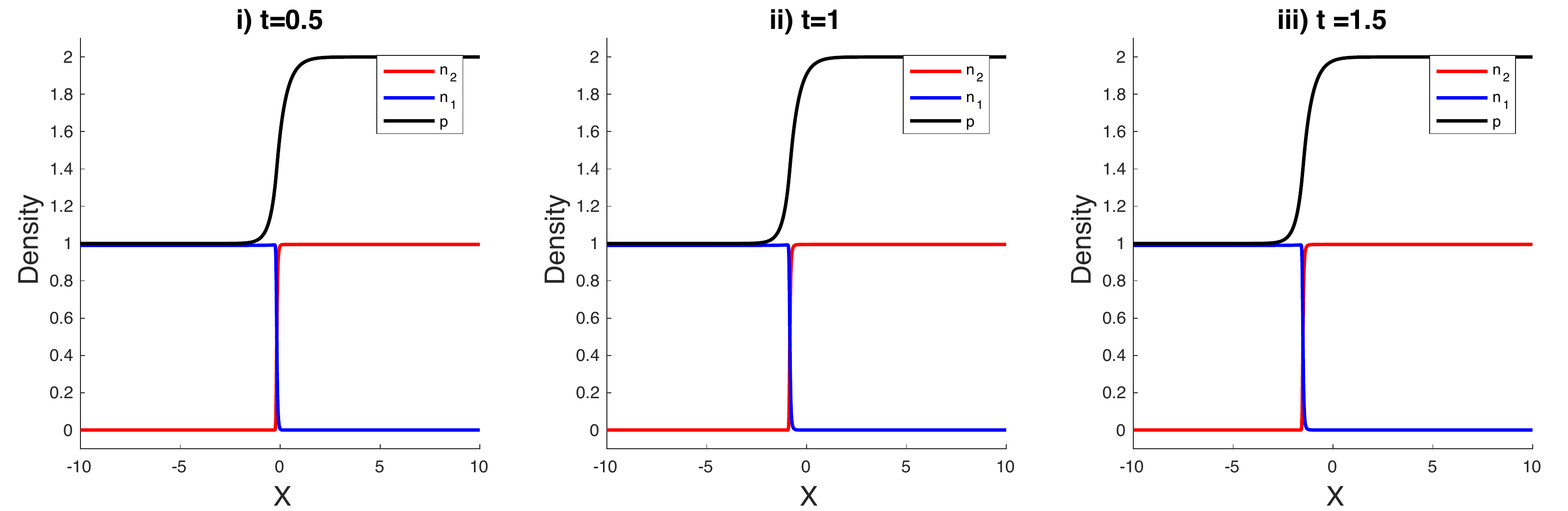}
}
\end{subfigure}
\hfill
 \begin{subfigure}[$\epsilon=0.001$]{
\includegraphics[width=0.7\linewidth]{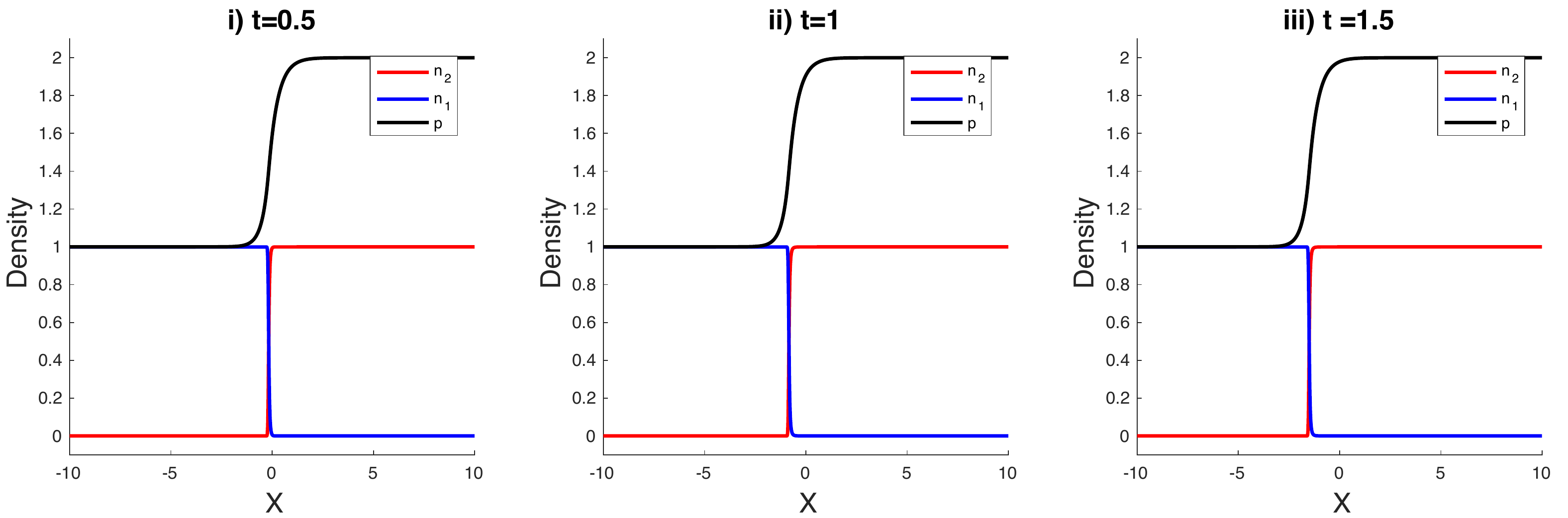}
}
\end{subfigure}
\hfill
  \begin{subfigure}[Hele-Shaw system]{
\includegraphics[width=0.7\linewidth]{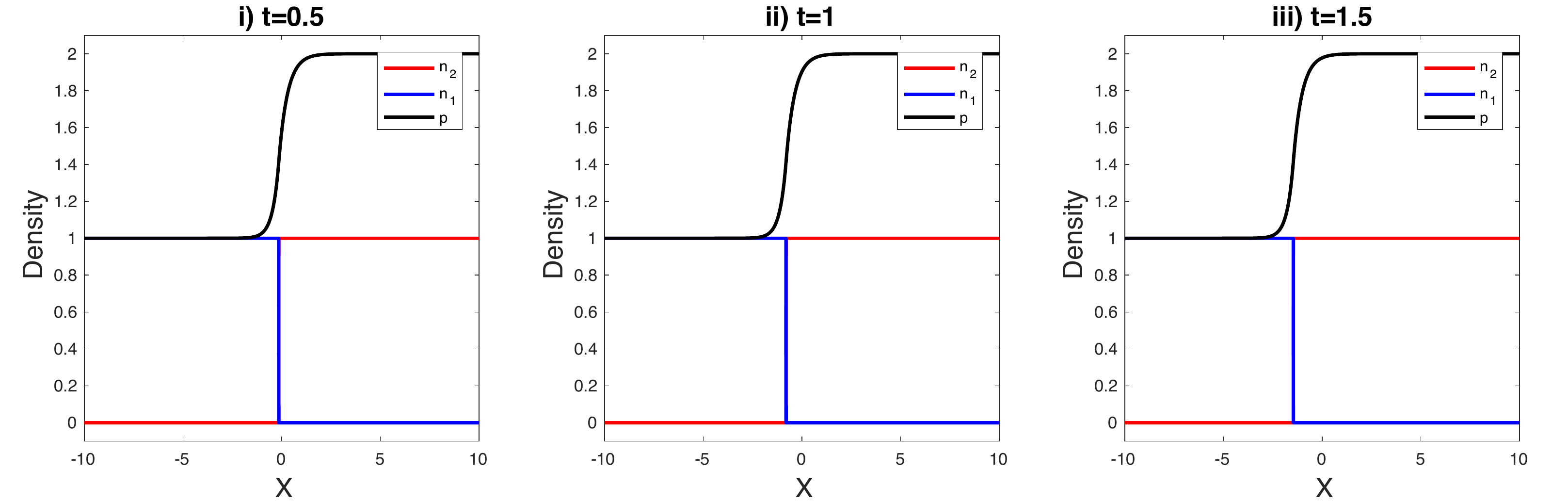}
}
\end{subfigure}
\caption{Densities $n_1$ (red), $n_2$ (blue) as functions of position $x$ at different times:  (i) $t=0.5$,   (ii) $t=1$,  (iii) $t=1.5$; and for different values of $\epsilon$: (a) $\epsilon=1$, (b) $\epsilon=0.1$, (c) $\epsilon=0.01$, (d) $\epsilon=0.001$, (e) Hele-Shaw system.}
\label{fig:2}
\end{figure}

We observe in Fig.~\ref{fig:2} that the time dynamics of the numerical solutions is similar for each case and follows the dynamics presented above for the case $\epsilon=1$. The main difference observed is the maximal packing value ${N_M^1}_{\epsilon}$ and ${N_M^2}_{\epsilon}$. Indeed since the maximal packing values are given by \eqref{P_M}, when $\epsilon \rightarrow 0$, the maximal packing value converges to $1$. This is consistent with the numerical results shown in Fig.~\ref{fig:2}. In addition we observe that as $\epsilon$ decreases the stiffness of the densities increases. In overall we observe that as $\epsilon \rightarrow 0$ densities converge to Heaviside functions.

\subsection{Particular solutions: tumor spheroid}

One interested application of this study is tissue development. Since we consider a system with two populations of cells, we can for example consider the case of tumour with proliferative cells, whose density is denoted $n_2$, and quiescent cells, whose density is denoted $n_1$.

\paragraph{Solution of the limiting Hele-Shaw problem.} We assume that initially the tumor is a spheroid centered in 0
and is composed by a spherical core representing the quiescent cells surrounded by a ring representing the proliferative cells. Then, we are looking for particular solution of the limiting Hele-Shaw problem \eqref{eq:n1}-\eqref{eq:n2} under the form:
$$ n_1(t,x)=\mathbf{1}_{\Omega_1(t)}(x) \quad \text{ with} \quad \Omega_1(t)=\{n_1(x,t)=1\}=B_{[-R_1(t),R_1(t)]},$$
$$ n_2(t,x)=\mathbf{1}_{\Omega_2(t)}(x) \quad \text{ with} \quad  \Omega_2(t)=\{n_2(x,t)=1\}={B_{(-L,L)}}\setminus B_{[-R_1(t),R_1(t)]}.$$
The radius $R_1(t)$, with $R_1(t)<L$, is computed according to the geometric motion rules
\begin{equation*}
  \left\{
   \begin{aligned}
        &R_1'(t)=-\partial_x p(R_1(t)),\\
        & R_1(0)=R_1^0,
      \end{aligned}
    \right.
\quad
\end{equation*}
where $p$ is the solution of 
$$ 
-\partial_{xx} p = n_1 G_1(p) +n_2 G_2(p)  \quad \text{in} \quad \Omega_1(t)\cup\Omega_2(t).
$$
Such functions $n_1$ and $n_2$ are solutions to the limiting Hele-Shaw problem \eqref{eq:n1}-\eqref{eq:n2}. Indeed by differentiating the densities, in the distributional sense, we get,
\begin{equation*}
\begin{aligned}
&\partial_t n_1= R_1'(t) (\delta_{x=R_1(t)}-\delta_{x=-R_1(t)}) , \\
& \partial_x (n_1 \partial_x p)=(\delta_{x=R_1(t)}-\delta_{x=-R_1(t)}) \partial_x p + \mathbf{1}_{ [-R_1(t),R_1(t)]} \partial_{xx} p.
\end{aligned}
\end{equation*}
Since $R_1'(t)= -\partial_{x}p(R_1(t))$, it follows that
$$ \partial_t n_1 - \partial_x (n_1 \partial_x p) =  \mathbf{1}_{ [-R_1(t),R_1(t)]} G_1(p) = n_1 G_1(p).$$
By applying the same computation on $n_2$ we get,
$$ 
\partial_t n_2 - \partial_x (n_2 \partial_x p) = n_2 G_2(p).
$$

\paragraph{Analytical solution.} As this paper is reduced to the case of dimension 1, we can compute the exact solution of the limiting Hele-Shaw problem \eqref{eq:n1}-\eqref{eq:n2} with this initial configuration for some simple expression of the growth terms $G_1$ and $G_2$. For instance, let us suppose that the growth terms are linear,
$$ 
G_1(p)=g_1(P_M^1-p) \quad \mbox{ and } \quad G_2(p)=g_2(P_M^2-p). 
$$
This choice means that as the pressure increases, the tumor will grow more slowly, until the pressure reach a critical value ($P_M^1$ or $P_M^2$ depending of the species) where the growth rate takes negative values, modelling the apoptosis of cells. The solution of the pressure equation is given by,
$$ 
p(x,t) = \left\{
    \begin{array}{ll}
         (P_M^1-P_M^2) \frac{\sqrt{g_2} \sinh(\sqrt{g_2}(R_1(t)-L))\cosh(\sqrt{g_1}x)}{\lambda}  & \quad \text{on } \Omega_1(t), \\
         (P_M^1-P_M^2) \frac{\sqrt{g_1} \cosh(\sqrt{g_2}(x-L))\sinh(\sqrt{g_1}R_1(t))}{\lambda} & \quad \text{on } \Omega_2(t).
    \end{array}
\right.
$$
with 
\begin{align*}
&\lambda= \sqrt{g_1}  \cosh(\sqrt{g_2}(R_1-L))\sinh(\sqrt{g_1}R_1) - \sqrt{g_2}  \sinh(\sqrt{g_2}(R_1-L))\cosh(\sqrt{g_1}R_1), 
\end{align*}
Computing the derivatives at the interface $R_1(t)$ we deduce that,
\begin{equation} \label{eqR_x}
          R'_1(t)=  -\sqrt{g_1 g_2} (P_M^1-P_M^2)  \frac{\sinh(\sqrt{g_2}(R_1(t)-L))\cosh(\sqrt{g_1}R_1(t))}{\lambda} .
\end{equation}
We are interested in the study of the evolution of $R_1$ in time, in function of the parameters $g_1, g_2, P_M^1, P_M^2$. Given that $0 \leq R_1(t) \leq L$, it is straightfoward that $\lambda \leq 0$. From \eqref{eqR_x}, we deduce that the sign of $R_1'(t) \geq 0$ is the same as the sign of $P_M^1-P_M^2$.

\paragraph{Numerical simulations} Finally we show some simulations of the mechanical problem for the case of spheroid tumor growth. We run the simulations with $\epsilon=0.01$ as we have shown in Section \ref{sec:numeps} that the simulations are close enough from the free boundary model. We consider two populations with the same space configuration as at the beginning of this section,
$$ n_1=0.5\, \mathbf{1}_{B_{[-R_1(t),R_1(t)]}} \quad \text{ and } \quad n_2=0.5\, \mathbf{1}_{B_{[-L,L]\setminus [-R_1(t),R_1(t)]}},$$
with
$$ R_1(0) = 0.5\quad \mbox{ and } \quad R_2(0) = 1.5.$$ 
We fix the parameter $\epsilon$ to the value 1.
The growth rates are going to defined the dynamics of the two populations. In the first example, we choose growth functions such that we observe death of the inner species $n_1$, which corresponds to the apoptosis of one population of cells.
The growth functions are defined by
\begin{equation}\label{GRfig1}
 G_1(p)=10(1-p)  \quad \mbox{ and } \quad G_2(p)=10(1-p/2),
\end{equation}
In a second example we display an example where the species $n_1$ grows and pushes the surrounding species $n_2$.
\begin{equation}\label{GRfig3}
 G_1(p)=10(4-p)  \quad \mbox{ and } \quad G_2(p)=10(1-p/2).
\end{equation} 
 
 \begin{figure}[!ht]
    \centering
  \begin{subfigure}[Case 1]{
 \includegraphics[width=1.\linewidth]{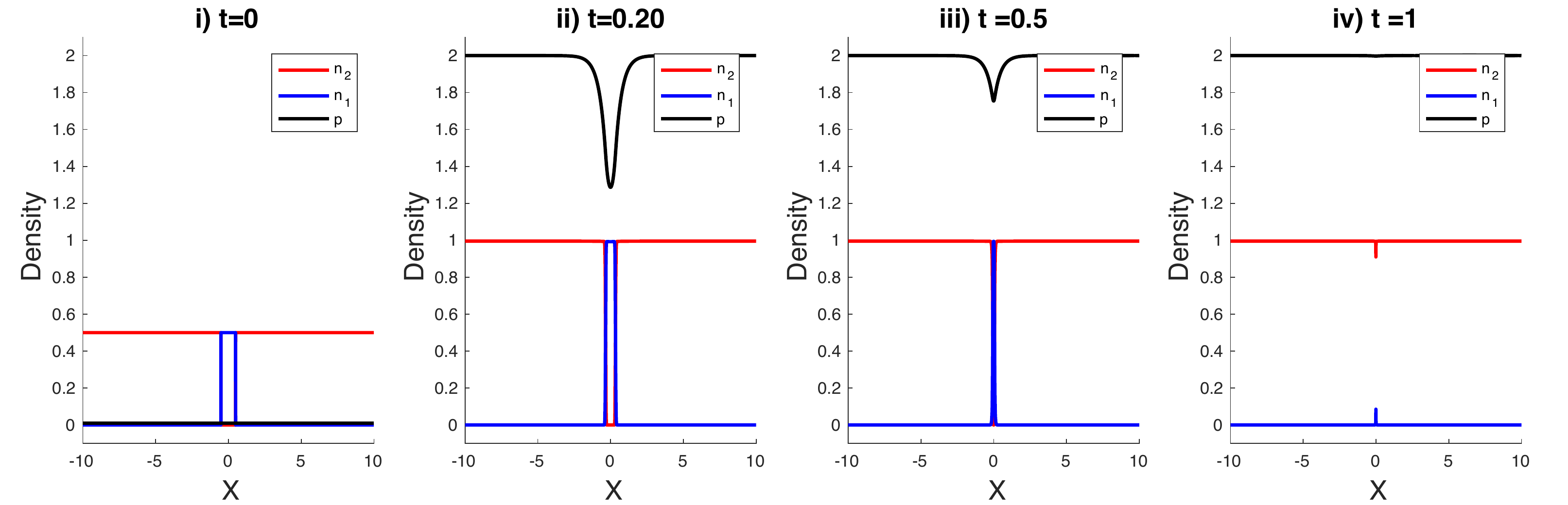}
 }
 \end{subfigure}
 \hfill
    \begin{subfigure}[Case 2]{
 \includegraphics[width=1.\linewidth]{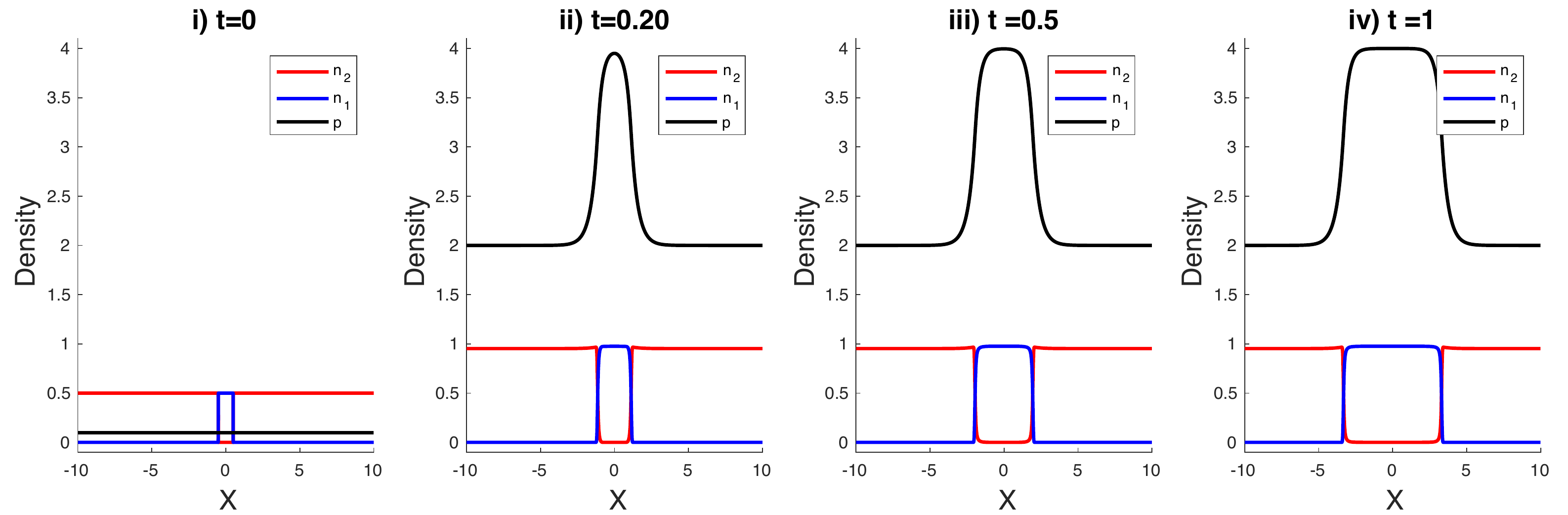}
 }
 \end{subfigure}
 \hfill 
 \caption{Densities $n_1$ (red), $n_2$ (blue) and $p$ (black) as functions of position $x$ for different growth function at different times: (i) $t= 0.3$, (ii) $t = 0.6$, (iii) $t = 1$, (iv) $t = 1.5$.}
 \label{fig:3}
 \end{figure}

In Fig \ref{fig:3}, we display the time dynamics of the densities of these two examples at different time step:  (i) $t= 0$, (ii) $t = 0.1$, (iii) $t = 0.3$, (iv) $t = 0.6$, (v) $t = 1$. It illustrates the two different behaviours mentionned above by \eqref{GRfig1} and \eqref{GRfig3}.
In  Fig \ref{fig:3} (a) the red species grows and the blue species disappears since the pressure in the domain is bigger that $P_M^1$.
In Fig \ref{fig:3} (b), the blue species pushes the red species and propagates.

\newpage

\begin{appendix}

\section{Uniqueness of solutions: Regularized dual problem}\label{appen}

In this appendix we prove rigorously Proposition \ref{TH2} using a regularization procedure for the dual problem \ref{dualpb}.
We follow closely the ideas in \cite[p 109-110]{PQV} which are recall here for the sake of completness of this paper.
Since the coefficients $A_i$, $B_i$ are not strictly positive and not smooth, then we need to regularize the problem \ref{dualpb}. For $i=1,2$, let ${A^k_i} $,  ${B^k_i} $, ${C^k_i} $ and  ${G^k_i}$ be sequences of smooth functions such that,
\begin{equation*}
  \left\{
      \begin{aligned}
         & \| A_i - A_i^k\|_{L^2(Q_T)} < \frac{\alpha_i}{k} , & \frac{1}{k}<A_i^k\leq1,  &&\\
        & \| B_i - B_i^k\|_{L^2(Q_T)} < \frac{\beta_i}{k} , & \frac{1}{k}<B_i^k\leq1, & &\\
        & \| C_i - C_i^k\|_{L^2(Q_T)} < \frac{\delta_{1,i}}{k} , & 0 \leq C_i^k\leq M_{1,i}, &\quad  \| \partial_t C_i^k \|_{L^1(Q_T)} \leq K_{1,i}, &\\
        & \| G_i(q_i) - G_i^k\|_{L^2(Q_T)} < \frac{\delta_{2,i}}{k} , & |G_i^k| < M_{2,i}, & \quad \| \partial_x G_i^k \|_{L^2(Q_T)} \leq K_{2,i} ,&\\
      \end{aligned}
    \right.
\end{equation*}
for some constant $\alpha_i, \beta_i, \delta_{1,i}, \delta_{2,i}, M_{1,i}, M_{2,i}, K_{1,i}, K_{2,i} $.
For any smooth function $\Phi_i$, $i=1,2$, we consider the following regularised dual system,
\begin{equation}\label{regdualpb}
\left\{\begin{array}{l}
\pa_t \psi^k_i + \frac{B^k_i}{A^k_i} \pa_{xx} \psi^k_i + G^k_i \psi^k_i - C^k_i \frac{B^k_i}{A^k_i}\psi^k_i = \Phi_i, \quad \mbox{ in } \ Q_T,  \\[1mm]
\partial_x\psi^k_i(\pm L) = 0 \quad \mbox{ in } \ (0,T), \qquad
\psi^k_i(\cdot,T) = 0 \quad \mbox{ in } (-L,L).
\end{array}\right.
\end{equation}
As the coefficients $\frac{B^k_i}{A^k_i}$ for $i=1,2$, are positive, continuous and bounded below away from zero, the dual equation is uniformly parabolic in $Q_T$. Then we can solve it and we denote $\psi^k_i$ the solution of \eqref{regdualpb}. This solution $\psi^k_i$ is smooth and can be used as a test function in \eqref{eq:uniq2}.

Using \eqref{eq:uniq2} and \eqref{regdualpb}, for $i=1,2$,
\begin{equation*}
\iint_{Q_T} (n_{i0}-\widetilde{n_{i0}}) \Phi_i \,dxdt = I_{1,i}- I_{2,i}- I_{3,i}+ I_{4,i},
\end{equation*}
where
\begin{equation*}
      \begin{aligned}
         & I_{1,i} =\iint_{Q_T}( {n_{i0}}-\widetilde{n_{i0}}+q_i - \widetilde{q_i}) \frac{B^k_i}{A^k_i} (A_i- A_i^k) (\Delta \psi^k_i- C_i^k \psi^k_i ) \,dxdt,\\
         & I_{2,i} =\iint_{Q_T}( {n_{i0}}-\widetilde{n_{i0}}+q_i - \widetilde{q_i}) (B_i- B_i^k) (\Delta \psi^k_i- C_i^k \psi^k_i ) \,dxdt, \\
         & I_{3,i} =  \iint_{Q_T}(n_{i0}-\widetilde{n_{i0}}) (G_i(q_i)-G_{i}^k)\psi^k_i \,dxdt, \\
         & I_{4,i} = \iint_{Q_T}( {n_{i0}}-\widetilde{n_{i0}}+q_i - \widetilde{q_i}) B_i(C_i- C_i^k) \psi^k_i \,dxdt .
      \end{aligned}
\end{equation*}
We intend to show that at the limit $k \rightarrow +\infty$, $I_{j,i}$ converges to 0 for $j=1,2,3,4$ and $i=1,2$. To show the convergence, we are going to find estimates on $\psi^k_i$ and its derivative:
\begin{itemize}
\item As $\psi^k_i$ is solution of \eqref{regdualpb} with $C_i^k $ nonnegative and $G_{i}^k$ uniformly bounded, from the maximum principle we get,
$$ \| \psi^k_i \|_{L^{\infty}(Q_T)} \leq \kappa_1 ,$$
where $ \kappa_1 $ is independent of $k$.
\item Multipling \eqref{regdualpb} by $\partial_{xx} \psi^k_i- C_i^k \psi^k_i$ and integrating on $\Omega \times (t, T)$, we get
\begin{equation}
      \begin{aligned} \label{regdualprobineq}
       \frac{1}{2} \| \partial_x \psi^k_i (t)\|^2_{L^2(-L,L)} + \iint_{\Omega \times (t, T)} \frac{B^k_i}{A^k_i} |\partial_{xx} \psi^k_i- C_i^k \psi^k_i|^2 \,dxdt = - \int_{-L}^L (C_i^k \frac{ (\psi^k_i)^2}{2})(t) \,dx & \\
       + \iint_{(-L,L) \times (t, T)} \Big(-\partial_t C_i^k \frac{ (\psi^k_i)^2}{2}-G_i^k|\partial_x \psi^k_i|^2 - \psi^k_i \partial_x G_i^k \, \partial_x \psi^k_i + C_i^k  G_i^k (\psi^k_i)^2 & \\
       + \psi^k_i \partial_{xx} \Phi_i - \Phi_i C_i^k \psi^k_i \Big)\,dxdt & \\
       \leq K \left(1-t+ \int_t^T \| \partial_x \psi^k_i (s)\|^2_{L^2(-L,L)} ds\right), &
      \end{aligned}
      \end{equation}
      with $K$ a constant independent of $k$. By using Gronwall lemma we get the following bound,
      $$ \sup_{0 \leq t \leq T} \| \partial_x \psi^k_i \|_{L^2(Q_T)} \leq \kappa_2 ,$$
with $ \kappa_2 $ independent of $k$.
\item Using \eqref{regdualprobineq}, we get
$$ \| \Big(\frac{B^k_i}{A^k_i}\Big)^{1/2} (\partial_{xx} \psi^k_i- C_i^k \psi^k_i) \|_{L^2(Q_T)} \leq \kappa_3 ,$$
with $ \kappa_3 $ independent of $k$.
\end{itemize}
We use these bounds to prove the convergence of the integrals $I_{j,i}$ for $j=1,2,3,4$ and $i=1,2$. We get,
\begin{equation*}
      \begin{aligned}
          I_{1,i} & = \tilde{K} \iint_{Q_T} \frac{B^k_i}{A^k_i} |A_i- A_i^k| |\partial_{xx} \psi^k_i- C_i^k \psi^k_i | \,dxdt \leq \tilde{K} \| (\frac{B^k_i}{A^k_i}\Big)^{1/2} (A_i- A_i^k) \|_{L^2(Q_T)} \\
         & \leq  \tilde{K} k^{1/2} \| (A_i- A_i^k) \|_{L^2(Q_T)} \leq \tilde{K} \alpha k^{-1/2}\\
          I_{2,i} &= \tilde{K} \iint_{Q_T} |B_i- B_i^k| |\partial_{xx} \psi^k_i- C_i^k \psi^k_i | \,dxdt \leq \tilde{K}  \| (\frac{A^k_i k^{1/2}}{B^k_i}\Big)^{1/2} (B_i- B_i^k) \|_{L^2(Q_T)} \\
                   & \leq \tilde{K} k^{1/2} \| (B_i- B_i^k) \|_{L^2(Q_T)} \leq \tilde{K} \beta k^{-1/2},\\
          I_{3,i} &=  \iint_{Q_T}|n_{i0}-\widetilde{n_{i0}}| |G_1(q_1)-G_{i}^k| |\psi^k_i| \,dxdt \leq  \tilde{K} \| (G_i(q_i)- G_i^k) \|_{L^2(Q_T)} \leq  \tilde{K}  \frac{\delta_{2,i}}{n},\\
          I_{4,i} &= \tilde{K} \iint_{Q_T} B_i|C_i- C_i^k|| \psi^k_i| \,dxdt \leq \tilde{K} \| (C_i- C_i^k) \|_{L^2(Q_T)} \leq \frac{\tilde{K}}{n} .
      \end{aligned}
\end{equation*}
where $\tilde{K}$ is a contant independent of of $k$. It justifies that $\lim_{k \rightarrow +\infty} I_{j,i} =0 $ for $j=1,2,3,4$ and $i=1,2$. Then 
$$ 
\lim_{k \rightarrow +\infty} \iint_{Q_T} (n_{i0}-\widetilde{n_{i0}}) \Phi_i \,dxdt =0,
$$ 
for any smooth function $\Phi_i$ for $i=1,2$. This implies that ${n_{10}} = \widetilde{n_{10}}$ and ${n_{20}} = \widetilde{n_{20}}$.
Then, we deduce from \eqref{eq:uniq1},
\begin{equation*}
\iint_{Q_T} \Big[ (q_i - \widetilde{q_i}) \pa_{xx} \psi_i + n_{i0}( G_i(q_i) - G_i(\widetilde{q_i}))\psi_i\Big] \,dxdt = 0.
\end{equation*}
By using $\psi_i = q_i - \widetilde{q_i}$, we recover $q_i = \widetilde{q_i}$ for $i=1,2$. It concludes the proof.

\end{appendix}

%%%%%%%%%%%%%%%%%%%%%%%%%%%%%%%%%%%
%
%%%%%% BIBLIO %%%%%%%%%%%%%%%%%%%%%%
%
%%%%%%%%%%%%%%%%%%%%%%%%%%%%%%%%%%%%

%%%%%%%%%%%%%%%%%%%%%%%%%%%%%%%%%%%%%%%%%%%%%%%%%%%%%%%%%%%%%%%%%%%%%%%%%%%%%%%%

\bibliographystyle{abbrv}
\bibliography{Cell_sorting_segregation_v12_2}

\begin{thebibliography}{10}

\bibitem{AM04}
R.~Araujo and D.~McElwain.
\newblock A history of the study of solid tumour growth: the con- tribution of
  mathematical modelling.
\newblock {\em D.L.S. Bull. Math. Biol.}, 66(5):1039, 2004.

\bibitem{BPM}
M.~Bertsch, R.~{Dal Passo}, and M.~Mimura.
\newblock A free boundary problem arising in a simplified tumour growth model
  of contact inhibition.
\newblock {\em Interfaces Free Bound.}, 12:235--250, 2010.

\bibitem{BERTSCH198756}
M.~Bertsch, M.~Gurtin, and D.~Hilhorst.
\newblock On a degenerate diffusion equation of the form c(z)t = $\phi$(zx)x
  with application to population dynamics.
\newblock {\em J. Differ. Equ.}, 67(1):56 -- 89, 1987.

\bibitem{B1987}
M.~Bertsch, M.~Gurtin, and D.~Hilhorst.
\newblock On interacting populations that disperse to avoid crowding: the case
  of equal dispersal velocities.
\newblock {\em Nonlinear Anal. Theory Methods Appl.}, 11(4):493 -- 499, 1987.

\bibitem{BGHP}
M.~Bertsch, M.~E. Gurtin, D.~Hilhorst, and L.~A. Peletier.
\newblock On interacting populations that disperse to avoid crowding: the
  effect of a sedentary colony.
\newblock {\em Q. Appl. Math.}, 19(1):1--12, 1984.

\bibitem{BHIM}
M.~Bertsch, D.~Hilhorst, H.~Izuhara, and M.~Mimura.
\newblock A non linear parabolic-hyperbolic system for contact inhibition of
  cell growth.
\newblock {\em Differ. Equ. Appl.}, 4(1):137--157, 2010.

\bibitem{BCGRS}
D.~Bresch, T.~Colin, E.~Grenier, B.~Ribba, and O.~Saut.
\newblock Computational modeling of solid tumor growth: the avascular stage.
\newblock {\em SIAM J. Sci. Comput.}, 32(4):2321--2344, 2010.

\bibitem{BT}
S.~N. Busenberg and C.~C. Travis.
\newblock Epidemic models with spatial spread due to population migration.
\newblock {\em J. Math. Biol.}, 16(2):181--198, 1983.

\bibitem{Byrne}
H.~Byrne and M.~Chaplain.
\newblock Growth of necrotic tumors in the presence and absence of inhibitors.
\newblock {\em Math. Biosci.}, 135(2):187 -- 216, 1996.

\bibitem{BD}
H.~Byrne and D.~Drasdo.
\newblock Individual-based and continuum models of growing cell populations: a
  comparison.
\newblock {\em J. Math. Biol.}, 58(4):657, 2008.

\bibitem{CarilloSchmidtchen}
A.~J. Carrillo, S.~Fagioli, F.~Santambrogio, and M.~Schmidtchen.
\newblock Splitting schemes \& segregation in reaction-(cross-) diffusion
  systems.
\newblock {\em preprint arXiv:1711.05434}, 2017.

\bibitem{CCH}
J.~A. Carrillo, A.~Chertock, and Y.~Huang.
\newblock A finite-volume method for nonlinear nonlocal equations with a
  gradient flow structure.
\newblock {\em Commun. Comput. Phys.}, 17(1):233--258, 2015.

\bibitem{CGP}
M.~Chaplain, L.~Graziano, and L.~Preziozi.
\newblock Mathematical modelling of the loss of tissue compression
  responsiveness and its role in solid tumour development.
\newblock {\em Math Med Biol.}, 23(3):197--229, 2006.

\bibitem{CDHV}
A.~Chertock, P.~Degond, S.~Hecht, and J.-P. Vincent.
\newblock Incompressible limit of a continuum model of tissue growth with
  segregation for two cell populations.
\newblock {\em preprint arXiv:1804.04090}, 2018.

\bibitem{BenAmar}
P.~Ciarletta, L.~Foret, and M.~Ben~Amar.
\newblock The radial growth phase of malignant melanoma: multi-phase modelling,
  numerical simulations and linear stability analysis.
\newblock {\em J. R. Soc. Interface}, 8(56):345--368, 2011.

\bibitem{Cui}
S.~Cui and J.~Escher.
\newblock Asymptotic behaviour of solutions of a multidimensional moving
  boundary problem modeling tumor growth.
\newblock {\em Comm. Partial Differential Equations}, 33(4):636--655, 2008.

\bibitem{FriedmanHu}
A.~Friedman and B.~Hu.
\newblock Stability and instability of {L}iapunov-{S}chmidt and {H}opf
  bifurcation for a free boundary problem arising in a tumor model.
\newblock {\em Trans. Am. Math. Soc.}, 360(10):5291--5342, 2008.

\bibitem{GALIANO}
G.~Galiano.
\newblock On a cross-diffusion population model deduced from mutation and
  splitting of a single species.
\newblock {\em Comput. Math. Appl.}, 64(6):1927 -- 1936, 2012.

\bibitem{GSV}
G.~Galiano, S.~Shmarev, and J.~Velasco.
\newblock Existence and multiplicity of segregated solutions to a cell-growth
  contact inhibition problem.
\newblock {\em Discrete Contin. Dyn. Syst.}, 35(4):1479--1501, 2015.

\bibitem{Greenspan}
H.~P. Greenspan.
\newblock Models for the growth of a solid tumor by diffusion.
\newblock {\em Stud. Appl. Math.}, 51(4):317--340, 1972.

\bibitem{HV}
S.~Hecht and N.~Vauchelet.
\newblock Incompressible limit of a mechanical model for tissue growth with
  non-overlapping constraint.
\newblock {\em Commun. Math. Sci.}, 15(7):1913--1932, 2017.

\bibitem{KimPozar}
I.~Kim and N.~Po{\v z}{\'a}r.
\newblock Porous medium equation to {H}ele-{S}haw flow with general initial
  density.
\newblock {\em Trans. Amer. Math. Soc.}, 370:873--909, 2018.

\bibitem{Lotka}
A.~J. Lotka.
\newblock Contribution to the theory of periodic reactions.
\newblock {\em J. Chem. Biol. Phys.}, 14(3):271--274, 1909.

\bibitem{MPQ}
A.~Mellet, B.~Perthame, and F.~Quir\'os.
\newblock A {H}ele-{S}haw problem for tumor growth.
\newblock {\em J. Funct. Anal.}, 273(10):3061--3093, 2017.

\bibitem{Mimura1980}
M.~Mimura and K.~Kawasaki.
\newblock Spatial segregation in competitive interaction-diffusion equations.
\newblock {\em J. Math. Biol.}, 9(1):49--64, 1980.

\bibitem{PQTV}
B.~Perthame, F.~Quir\`os, M.~Tang, and N.~Vauchelet.
\newblock Derivation of a {H}ele-{S}haw type system from a cell model with
  active motion.
\newblock {\em Interfaces Free Bound.}, 14(4):489--508, 2014.

\bibitem{PQV}
B.~Perthame, F.~Quir{\'o}s, and J.~L. V{\'a}zquez.
\newblock The {H}ele--{S}haw asymptotics for mechanical models of tumor growth.
\newblock {\em Arch. Ration. Mech. Anal.}, 212(1):93--127, 2014.

\bibitem{PV}
B.~Perthame and N.~Vauchelet.
\newblock Incompressible limit of a mechanical model of tumour growth with
  viscosity.
\newblock {\em Philos. Trans. Roy. Soc. A, Math. Phys. Eng. Sci.},
  373(2050):20140283, 2015.

\bibitem{RBEJPJ}
J.~Ranft, M.~Basan, J.~Elgeti, J.-F. Joanny, J.~Prost, and F.~J{\"u}licher.
\newblock Fluidization of tissues by cell division and apoptosis.
\newblock {\em Proc. Natl. Acad. Sci.}, 107(49):20863--20868, 2010.

\bibitem{SHIGESADA197983}
N.~Shigesada, K.~Kawasaki, and E.~Teramoto.
\newblock Spatial segregation of interacting species.
\newblock {\em J. of Theor. Biol.}, 79(1):83 -- 99, 1979.

\bibitem{simon}
J.~Simon.
\newblock Compact sets in the space {$L^p(0,T;B)$}.
\newblock {\em Ann. Mat. Pura Appl. (4)}, 146:65--96, 1987.

\end{thebibliography}

\end{document}